\documentclass[12pt,reqno]{amsart}
\usepackage{a4wide}
\usepackage{mathrsfs} 
\usepackage{hyperref}
\usepackage{stmaryrd}

\usepackage{amssymb}

\theoremstyle{plain}
\newtheorem{theorem}{Theorem}[section]
\newtheorem{lemma}[theorem]{Lemma}
\newtheorem{sublemma}[theorem]{Sublemma}
\newtheorem{proposition}[theorem]{Proposition}

\theoremstyle{definition}
\newtheorem{definition}[theorem]{Definition}

\theoremstyle{remark}
\newtheorem{remark}[theorem]{Remark}

\newcommand{\Cc}{\mathbb{C}}

\newcommand{\Ee}{\mathbb{E}}

\newcommand{\Nn}{\mathbb{N}}
\newcommand{\Pp}{\mathbb{P}}

\newcommand{\Rr}{\mathbb{R}}

\newcommand{\Uu}{\mathbb{U}}

\newcommand{\Zz}{\mathbb{Z}}

\newcommand{\Un}{1\!\!1}  

\newcommand{\Eee}{\mathcal{E}}
\newcommand{\Fe}{\mathcal{F}}

\newcommand{\Me}{\mathcal{M}}

\newcommand{\Es}{\mathscr{E}}

\newcommand{\Gs}{\mathscr{G}}

\newcommand{\Ns}{\mathscr{N}}

\newcommand{\Ys}{\mathscr{Y}}

\newcommand{\Arg}{\operatorname{Arg}}
\newcommand{\trnp}{\operatorname{tr}}
\newcommand{\card}{\operatorname{card}}

\renewcommand{\Re}{\operatorname{Re}}
\renewcommand{\Im}{\operatorname{Im}}

\renewcommand{\geq}{\geqslant}
\renewcommand{\leq}{\leqslant}

\let\oldforall\forall
\def\forall{\oldforall\,}

\newcommand{\ensemble}[1]{ \left\lbrace #1 \right\rbrace } 
\newcommand{\prth}[1]{\left( #1 \right) } 
\newcommand{\Esp}[1]{ \Ee \prth{ #1 } } 
\newcommand{\Espr}[2]{ \Ee_{#1} \! \prth{ #2 } } 
\newcommand{\Prob}[1]{ \Pp \prth{ #1 } } 
\newcommand{\Proba}[2]{ \Pp_{#1} \prth{ #2 } } 
\newcommand{\crochet}[1]{\left[ #1 \right] } 
\newcommand{\intcrochet}[1]{\llbracket #1 \rrbracket} 
\newcommand{\abs}[1]{\left| #1 \right|} 
\newcommand{\Bigabs}[1]{\Big| #1  \Big|} 
\newcommand{\biggabs}[1]{\bigg| #1  \bigg|} 
\newcommand{\Unens}[1]{ \Un_{ \ensemble{#1} } } 
\newcommand{\tr}[1]{\trnp{ \prth{ #1 } } } 

\newcommand{\intpi}[1]{\int_0^{2 \pi} #1 \frac{d \theta}{2 \pi} }

\newcommand{\intmean}[3]{\int_0^{#2} #3 \frac{d #1}{#2} }

\begin{document}


\title[Sums of characteristic polynomials of unitary matrices]{On the number of zeros of linear combinations of independent characteristic polynomials of random unitary matrices}


\author[Y. Barhoumi]{Yacine Barhoumi}
\address{Institut f\"ur Mathematik, Universit\"at Z\"urich, Winterthurerstrasse 190, 8057-Z\"urich, Switzerland}
\email{yacine.barhoumi@math.uzh.ch}

\author[C.P. Hughes]{Christopher Hughes}
\address{Department of Mathematics, University of York, Heslington, York, UK YO10 5DD, United Kingdom}
\email{christopher.hughes@york.ac.uk}

\author[J. Najnudel]{Joseph Najnudel}
\address{Institut de Math\'ematiques de Toulouse, Universit\'e Paul Sabatier, 118 route de Narbonne, F-31062 Toulouse Cedex 9, France }
\email{joseph.najnudel@math.univ-toulouse.fr}

\author[A. Nikeghbali]{Ashkan Nikeghbali}
\address{Institut f\"ur Mathematik, Universit\"at Z\"urich, Winterthurerstrasse 190, 8057-Z\"urich, Switzerland}
\email{ashkan.nikeghbali@math.uzh.ch}


\date{\today}


\begin{abstract}

We show that almost all the zeros of any finite linear combination  of independent characteristic polynomials of random unitary matrices lie on the unit circle. This result is the random matrix  analogue of an earlier result by Bombieri and Hejhal on the distribution of zeros of linear combinations of $L$-functions, thus providing further evidence for the conjectured links between the value distribution of the characteristic polynomial of random unitary matrices and the value distribution of $L$-functions on the critical line.

\end{abstract}

\maketitle

\section{Introduction}
Over the past two decades, there have been many new results at the interface of random matrix theory and analytic number theory that can be considered as evidence for the zeros of the Riemann zeta function being statistically distributed as eigenvalues of large random matrices (GUE matrices or Haar distributed unitary matrices); the interested reader can refer to \cite{MS}, \cite{KaS2} and \cite{RS} for a detailed account with many references, and to \cite{KaS} for the function field framework. Since the seminal papers by Keating and Snaith \cite{KS, KS2}, it is believed that the characteristic polynomial of random unitary  matrices on the unit circle models very accurately the value distribution of the Riemann zeta function (or more generally $L$-functions) on the critical line. This analogy was used by Keating and Snaith to produce the moments conjecture and since then the characteristic polynomial has been the topic of many research papers, and the moments of the characteristic polynomial have now been derived with many different methods, e.g.\ representation theoretic methods (see \cite{BumGa, POD}), super-symmetry method (see \cite{MS}), analytic methods (Toeplitz determinant methods as explained in the lecture by E. Basor in \cite{MS}, orthogonal polynomials on the unit circle method  \cite{KiNe}) or probabilistic methods (\cite{BHNY}), each method bringing a new insight to the problem. Many more fine properties of the characteristic polynomial have been established (e.g. large deviations principle in \cite{HKO}, local limit theorems \cite{KN}, the analogue of the moments conjecture for finite field zeta functions \cite{JKN}, etc.). Moreover, thanks to this analogy, one has been able to perform calculations in the random matrix world (whose analogue in the number theory world seems currently out of reach) to produce conjectures for the analogue arithmetic objects (see \cite{S} for a recent account).

There are nonetheless certain results that can be proved in both sides, such as Selberg's central limit theorem for the Riemann zeta function and the Keating-Snaith central limit theorem for the characteristic polynomial of random unitary matrices (see \cite{KS}). In fact Selberg's central limit theorem can be proved more generally for a wide class of $L$-functions (see \cite{Se} and \cite{BH}). Roughly speaking, an $L$-function must be defined by a Dirichlet series for $\Re(s)>1$, have an Euler product (with some growth condition on the coefficients of this product), an analytic continuation (except for finitely many poles all located on the line $\Re (s)=1$), and must satisfy a functional equation.  Such $L$-functions are expected to satisfy the general Riemann hypothesis (GRH), which says that all the non-trivial zeros are located on the critical line, the line $\Re(s)=1/2$.

Now if one considers a finite number of such $L$-functions, satisfying the same functional equation, then one can wonder if the zeros of a linear combination of these $L$-functions are still on the critical line. The answer is in general that GRH does not hold anymore for such a linear combination even though it still has a functional equation (this can be thought of coming from the fact that such a linear combination does not have an Euler product anymore). But Bombieri and Hejhal proved in \cite{BH} that nonetheless $100\%$ of the zeros of such linear combinations are still on the critical line (under an extra assumption of ``near orthogonality'' which ensures that the $\log$ of the $L$-functions are statistically asymptotically independent). In this paper we will show that a similar result holds for linear combinations of independent characteristic polynomials of random unitary matrices. The result on the random matrix side is technical and difficult and besides being an extra piece of evidence that the characteristic polynomial is a good model for the value distribution of $L$-functions, the result is also remarkable when viewed in the general setting of random polynomials as we shall explain it. The main goal of this article is to show that on average, any linear combination of characteristic polynomials of independent random unitary matrices has a proportion of zeros on the unit circle which tends to $1$ when the dimension goes to infinity.

More precisely, if $U$ is a unitary matrix of order $N \geq 1$, let $\Phi_U$ be the characteristic polynomial of $U$, in the following sense: for $z\in\Cc$,
$$\Phi_U(z) =  \det\prth{I_N - z U}.$$ From the fact that $U$ is unitary, we get the functional equation: $$\Phi_U(z) = (-z)^N \det (U) \overline{\Phi_U}(1/z).$$
For $z$ on the unit circle, this equation implies that $$\Phi_U(z) = R(z) \sqrt{(-z)^N \det(U)},$$ where $R(z)$ is real-valued (with any convention taken for the square root).
The fact that $\Phi_U$ has many zeros (in fact, all of them) on the unit circle can be related to the fact that the condition needed for $\Phi_U$ to vanish is in only unidimensional
(i.e. $R(z) = 0$ for a real-valued function $R$). Now, let $(U_j)_{1 \leq j \leq n}$ be unitary matrices of order $N$, and let $(b_j)_{1 \leq j \leq n}$ be real numbers: we wish to study
the number  of  zeros on the unit circle of the linear combination
 $$F_N = \sum_{j = 1}^n b_j \Phi_{U_j}. $$
 If we want that $F$ has most of its zeros on the unit circle, it is reasonable to expect that we need a ``unidimensional condition'' for the equation $F(z) = 0$
if $|z| = 1$, i.e. a functional equation similar to the equation satisfied by $U$. This equation obviously exists if all the characteristic polynomials $\Phi_{U_j}$ satisfy the {\it same} functional
equation, i.e. the matrices $U_j$ have the same determinant. By symmetry of the unitary group, it is natural to assume that the unitary matrices have determinant $1$. More precisely, the
main result of the article is the following:
\begin{theorem} \label{main}
 Let $(b_j)_{1 \leq j \leq n}$ be a family of (deterministic) real numbers, different from zero. For $N \geq 1$, let
 $$F_N := \sum_{j=1}^n b_j \Phi_{U_{N,j}},$$
where $(U_{N,j})_{1 \leq j \leq n}$ is a family of independent matrices following the Haar measure on the special unitary group $SU(N)$.
Then, the expected proportion of zeros of $F_N$ on the unit circle tends to $1$ when $N$ goes to infinity, i.e.
$$\Esp{ | \{z \in \Uu, F_N(z) = 0\}|} = N - o(N),$$
where $| \{z \in \Uu, F_N(z) = 0\}|$ is the number of $z$ on the unit circle which satisfy $ F_N(z) = 0$.
\end{theorem}
The whole paper is devoted to the proof of this result. Before explaining the strategy of the proof, we make a few remarks.
\begin{remark}
Theorem \ref{main} can be stated as
$$\lim_{n\to\infty} \Esp{\dfrac{1}{N} | \{z \in \Uu, F_N(z) = 0\}|} = 1.$$ Since the random variable $\dfrac{1}{N} | \{z \in \Uu, F_N(z) = 0\}|$ is bounded by $1$, in fact the convergence holds in all $L^p$ spaces for $p\geq1$. It also holds in probability since convergence in $L^1$ implies convergence in probability.
\end{remark}
\begin{remark}
The fact that we impose our matrices to have the same determinant is similar to the condition in \cite{BH} of the $L$-functions to have the same functional equation. Moreover, in our framework, the analogue of the Riemann hypothesis is automatically satisfied since all the zeros of each characteristic polynomial are on the unit circle.
\end{remark}
\begin{remark}
The fact that the proportion of zeros on the unit circle tends to $1$ is a remarkable fact as a result about random polynomials. Indeed it is well known that the characteristic polynomial of a unitary matrix is self-inversive (that is $a_{N-k}=\exp(i \theta) \bar{a}_k$ for some $\theta\in\Rr$, if $(a_k)_{0\leq k\leq N}$ are the coefficients of the polynomial). As explained in \cite{BBL}, self-inversive random polynomials are of interest in the context of semiclassical approximations in quantum mechanics and determining the proportion of zeros on the unit circle is there an important problem. Bogomolny, Bohigas and Leboeuf showed that if the first half of the coefficients (the second half being then fixed by the self-inverse symmetry) are chosen as independent complex Gaussian random variables, then asymptotically a fraction of $\frac{1}{\sqrt{3}}$ of the zeros are exactly on the unit circle. Hence we can say that our result is not typical of what is expected for classical random polynomials built from independent Gaussian random variables. In our framework, we do not even know the distribution of the coefficients and we also know that they are in fact not independent. Consequently the classical methods which use the independence of the coefficients (or the fact that they are Gaussian if one wants to add some dependence) would not work here.  Using general results on random polynomials whose coefficients are not independent and which do not have the same distribution as stated in \cite{HN}, one can deduce that the zeros cluster uniformly around the unit circle. But showing that they are almost all precisely on the unit circle is a much more refined statement.
\end{remark}
We now say a few words about our strategy of proof of Theorem \ref{main}. In fact we use the same general method as in \cite{BH}, called the "carrier waves" method, but the ingredients of our proof are different, in the sense that they are probabilistic: for instance we use the coupling method, concentration inequalities and the recent probabilistic representations of the characteristic polynomial obtained in \cite{BHNY}. More precisely, for $U \in U(N)$ and $t \in \Rr$, we denote by $Z_U(t)$ the characteristic polynomial of $U$ taken at $e^{-it}$, i.e. $Z_U(t) = \Phi_U(e^{-it})$.  Then we make a simple transformation of the linear combination $F_N$ in order that it is real valued when restricted as a function on the unit circle:

\begin{equation}\label{aserty}
i^N e^{iN \theta /2} F_N(e^{-i \theta}) = i^N e^{iN \theta /2} \sum_{j=1}^n b_j \Phi_{U_{j}} (e^{-i \theta}) = \sum_{j=1}^n b_j i^N e^{iN \theta/2} Z_{U_{j}} (\theta).
\end{equation}

Using the fact that $U_{j} \in SU(N)$, one checks that $i^N e^{iN \theta/2} Z_{U_{j}} (\theta)$ is real, and that then the number of zeros of $F_N$ on the unit circle is
bounded from below by the number of sign changes, when $\theta$ increases from $\theta_0$ to $\theta_0 + 2 \pi$ (with $\theta_0$ to be chosen carefully), of the real quantity given by the right-hand side of the equation above. The notion of carrier waves is explained in detail in \cite{BH}, p. 824--827 and we do not explain it again but we would rather give a general outline.
The main idea is that informally, with "high" probability and for "most" of the values of $\theta$,
one of the characteristic polynomials $Z_{U_j}$ dominates all the others (it is the "carrier wave").
More precisely, Lemma \ref{L8BH} implies the following: if $\delta$ depends only on $N$ and tends
to zero when $N$ goes to infinity, then there exists, with probability $1-o(1)$, a subset of
$[\theta_0,\theta_0 + 2\pi)$ with Lebesgue measure $o(1)$
such that for any $\theta$ outside this set, one can find $j_0$ between $1$ and $N$ such that $\log |Z_{U_{j_0}} (\theta)|  - \log |Z_{U_{j}} (\theta)| > \delta \sqrt{\log N}$
for all $j \neq j_0$.  In other words, one of the terms
in the sum of the right-hand side of \eqref{aserty} should dominate all the others. 
Moreover, Lemma \ref{speedgoodoscillation} informally gives the following: with high probability, the
order of magnitude of each of the characteristic polynomials does not change too quickly, and then, if the interval
$[\theta_0,\theta_0 + 2\pi)$ is divided into sufficiently many equal subintervals, the index of the carrier wave
remains the same in a "large" part of each subinterval.
Now, in an interval for which the carrier wave index $j_0$ remains the same, the zeros of
$Z_{U_{j_0}}$ correspond to sign changes of $i^N e^{iN \theta/2} Z_{U_{j_0}} (\theta)$, i.e. the dominant term
of $\eqref{aserty}$. Then, one gets sign changes of $i^N e^{iN \theta /2} F_N(e^{-i \theta})$, and by counting all
these sign changes, one deduces a lower
bound for the number of zeros of $F_N$ on the unit circle. The main issue of the present paper is to make rigorous this informal construction, in such a way that one gets
a lower bound $N - o(N)$.
 One of the reasons why the proof becomes technical and involved is that we have to take into account two different kinds of sets, and show that they have almost "full measure": subsets of the interval $[\theta_0,\theta_0+2\pi)$ and
  subsets of $SU(N)$.

More precisely, our proof is structured as follows. We first give two standard results (Propositions
 \ref{Disintegration21})  and \ref{propositionImLogCardZeros}), one on the disintegration of the Haar measure on $U(N)$ (indeed, most results on random matrices are established for $U(N)$ and we must find a way to go from the results for $U(N)$ to those for $SU(N)$) and the other one which establishes a relationship between the number of eigenvalues in a given fixed arc to the variation of the imaginary part of the $\log$ of the characteristic polynomial. Then we provide some estimates on the real and imaginary parts of the $\log$ of the characteristic polynomial (Lemmas
 \ref{logZX} and \ref{estimateimaginarypart}) as well as a bound on the concentration of the law of the log-characteristic polynomial (Lemma \ref{boundconcentration}). These estimates and some more intermediary one we establish are  also useful on their own and complete the existing results in the literature on the characteristic polynomial. Then we provide bounds on the oscillations of the real and imaginary parts of the $\log$ of the characteristic polynomial (Lemma \ref{L7BH}). We then introduce our subdivisions of the interval $[\theta_0,\theta_0+2\pi)$ and the corresponding relevant random sets to implement the carrier waves technique. Finally we combine all these estimates together to show that the average number of sign changes of (\ref{aserty}) is at least $N \left( 1- O\prth{ (\log N)^{-1/22} } \right)$ (the exponent $-1/22$ not being playing any major role  in our analysis).

\section*{Notation}
We gather here some notation used throughout the paper.

$U(N)$ stands for the unitary group of order $N$, while $SU(N)$ stands for the subgroup of elements  $U(N)$ whose determinant is equal to $1$. $\Pp_{U(N)}$ and $\Pp_{SU(N)}$ will denote the probability Haar measure on $U(N)$ and $SU(N)$ respectively. Similarly we denote by $ \Ee_{U(N)}$ and $ \Ee_{SU(N)}$ the corresponding expectations.

We shall denote the Lebesgue measure on $\Rr$ by $\lambda$. If $\alpha>0$ is a constant and if $I$ is an interval of length $\alpha$, then $\lambda_\alpha$ will denote the normalized measure $\frac{1}{\alpha} \lambda$ on the interval $I$.

If $n$ is an integer, we note $\intcrochet{1,n}$ the set of integers $\{1,\cdots,n\}$.

If $\Eee$ is a finite set, we note $|\Eee|$ the number of its elements.

For $n$ a positive integer, we note $ \Pp_{SU(N)}^{(n)}$ be the $n$-fold product of the Haar measure on $SU(N)$, and $ \Ee_{SU(N)}^{(n)}$ the corresponding expectation.

If $U$ is a unitary matrix, we note for $z\in\Cc$ its characteristic polynomial by $\Phi_U(z) =  \det\prth{I_N - z U}$. For  $t \in \Rr$, we denote by $Z_U(t)$ the characteristic polynomial of $U$ taken at $e^{-it}$, i.e. $Z_U(t) = \Phi_U(e^{-it})$.

We shall introduce several positive quantities in the sequel: $K>0$, $M>0$ and $\delta>0$. The reader should have in mind that these quantities will eventually depend on $N$. Unless stated otherwise, $N\geq4$ and $K$ is an integer such that $2\leq K\leq N/2$, and $M=N/K$. In the end we will use $K\sim N/ (\log N)^{3/64}$ and $\delta\sim (\log N)^{-3/32}$.

\section{Some general facts}

In this section, we state some general facts in random matrix theory, which will be used in the sequel.

\subsection{Disintegration of the Haar measure on unitary matrices}

\begin{proposition}\label{Disintegration21} %
Let $\Pp_{ U(N) }$ be the Haar measure on $U(N)$, $\Pp_{ SU(N) }$ the Haar measure on
$SU(N)$, and for $\theta \in \Rr$, let $\Pp_{SU(N), \theta}$ be the image of $\Pp_{ SU(N) }$ by the
application $U \mapsto e^{i \theta} U$ from $U(N)$ to $U(N)$. Then, we have the following
equality:

\vspace{-0.3cm}
\begin{eqnarray}\label{Disintegration}
\intpi{ \Pp_{SU(N), \theta }  \ \ } =  \Pp_{ U(N) },
\end{eqnarray}
i.e. for any continuous function $F$ from $U(N)$ to $\mathbb{R_+}$, the expectation
$\Ee_{SU(N), \theta } (F)$ of $F$ with respect to $\Pp_{SU(N), \theta }$ is
measurable with respect to $\theta$ and
$$\intpi{ \Ee_{SU(N), \theta } (F)  \ \ } =  \Ee_{ U(N)} (F) .$$
\end{proposition}
\begin{proof}
One has \begin{equation} \Ee_{SU(N), \theta } (F)
= \int F(X e^{i \theta}) d\Pp_{SU(N)}(X), \label{qwertyuiop}
\end{equation}
which, by dominated convergence, is continuous, and a fortiori measurable with respect to $\theta$.
By integrating  \eqref{qwertyuiop} with respect to  $\theta$, one sees that the proposition
 is equivalent to the following: if $U$ is a uniform matrix on $SU(N)$, and if $Z$ is independent, uniform on the unit circle, then $ZU$ is uniform on $U(N)$.
Now, let $A$ be a deterministic matrix in $U(N)$. For any $d \in \Cc$ such that $d^{-n} = \det(A)$, one has $Ad \in SU(N)$, and then
$ZUA  = (Z/d) (UAd)$, where:
\begin{enumerate}
\item $UAd$ follows the Haar measure on $SU(N)$ (since this measure is invariant by multiplication by $Ad \in SU(N)$).
\item $Z/d$ is uniform on the unit circle (since $d$, as $\det(A)$, has modulus $1$).
\item These two variables, which depend deterministically on the independent variables $A$ and $Z$, are independent.
\end{enumerate}
Hence $ZUA$ has the same law as $ZU$, i.e. this law is invariant by right-multiplication by any unitary matrix. Hence, $ZU$ follows the Haar measure on $U(N)$.
\end{proof}


\subsection{Number of eigenvalues in an arc:}

The result we state here relates the number of eigenvalues of a unitary matrix on a given arc to the logarithm of its characteristic polynomial.
For $U \in U(N)$ and $t \in \Rr$, we denote by $Z_U(t)$ the characteristic polynomial of $U$ taken at $e^{-it}$, i.e. $Z_U(t) = \Phi_U(e^{-it})$.
Moreover, if $e^{it}$ is not an eigenvalue of $U$ (which occurs almost surely under Haar measure on $U(N)$, and also under the Haar measure on $SU(N)$, except for
$e^{it} = 1$ and $N =1$), we define the logarithm of $Z_U(t)$, as follows:
\begin{equation}
\log Z_U(t) := \sum_{j = 1}^N \log( 1 -  e^{i(\theta_j-t)}), \label{definitionlogcharacteristic}
\end{equation}
where $\theta_1, \dots, \theta_N$ is the sequence of zeros of $Z_U$ in $[0, 2\pi)$, taken with multiplicity (notice that the eigenvalues of
$U$ are $e^{i\theta_1}, \dots, e^{i \theta_N}$), and where the principal branch of the logarithm is taken in the right-hand side.
We then have the following result, already stated, for example, in \cite{HKO}:
\begin{proposition} \label{propositionImLogCardZeros}
Let $0 \leq s < t < 2 \pi$, and let us assume that $s$ and $t$ are not zeros of $Z_U$. Then, the number of zeros of $Z_U$ in the interval $(s,t)$ is given as follows:	
\vspace{-0.3cm}
\begin{eqnarray}\label{ImLogCardZeros}
\sum_{k = 1}^N \Unens{\theta_k \in (s,t)} = \frac{N}{2\pi} (t-s) + \frac{1}{\pi} \prth{ \Im\log Z_U(t) - \Im\log Z_U(s) }.
\end{eqnarray}
\end{proposition}
\begin{proof}
It is sufficient to check that for all $\theta \in [0, 2 \pi) \backslash \{s,t\}$,
$$\pi \Unens{ \theta \in (s,t) } = \frac{t-s}{2}  + \Im \log\prth{ 1 - e^{ i (\theta - t) } } - \Im \log\prth{ 1 - e^{ i (\theta - s) } }.$$
Now, for $v \in (0,2 \pi)$,
$$  1 - e^{ iv} = e^{iv/2} ( e^{-iv/2} - e^{iv/2}) = -2 i \sin(v/2) \, e^{iv/2} = 2  \sin(v/2) \, e^{i(v-\pi)/2}.$$
Now, $\sin (v/2) > 0$ and $(v - \pi)/2 \in (-\pi/2, \pi/2)$ and hence
$$  \Im \log\prth{ 1 - e^{ i v} } = \frac{v - \pi}{2},$$
since we take the principal branch of the logarithm.
Now, for $\theta \in [0, 2 \pi) \backslash \{s,t\}$, $\theta - s + 2 \pi \Unens{\theta < s}$ and $\theta - t + 2 \pi \Unens{\theta < t}$ are in $(0,2 \pi)$, 
which implies
\begin{align*}
\Im \log\prth{ 1 - e^{ i (\theta - t) } } - \Im \log\prth{ 1 - e^{ i (\theta - s) } } & = \frac{\theta - t - \pi + 2 \pi \Unens{\theta < t}}{2} - \frac{\theta - s - \pi + 2 \pi \Unens{\theta < s}}{2}
\\ & = \frac{s - t}{2} + \pi \left( \Unens{\theta < t} - \Unens{\theta < s} \right), 
\end{align*}
and then Proposition \ref{propositionImLogCardZeros}.
\end{proof}

\section{Proof of Theorem \ref{main}}

\subsection{Conventions}

All the random matrices we will consider are defined, for some $N \geq 1$, on the measurable space $ (\Me_N(\Cc), \Fe)$, where $\Fe$ denotes the Borel $\sigma$-algebra of $\Me_N(\Cc)$.
The canonical matrix, i.e the random variable from $ (\Me_N(\Cc), \Fe)$ to $\Me_N(\Cc)$ defined by the identity function, is denoted $X$. Moreover, we denote by $\Ee_{U(N)}$ the expectation under
$\Pp_{U(N)}$, the Haar measure on $U(N)$, and by $\Ee_{SU(N)}$ the expectation under $\Pp_{SU(N)}$, the Haar measure on $SU(N)$. For example, if $F$ is a bounded, Borel function
from $\Me_N(\Cc)$ to $\mathbb{R}$,
$$\Ee_{SU(N)} [F(X)] = \int_{\Me_N(\Cc)} F(M) d \Pp_{SU(N)}(M).$$

\subsection{An estimate in average of the logarithm of the characteristic polynomial}

$ ~ $

\begin{lemma} \label{logZX} There exists a universal constant $ c_1 > 0 $ such that for all $N \geq 2$, and $A \geq 0 $,

\vspace{-0.3cm}
\begin{eqnarray*}
\intpi{\Proba{SU(N)}{ \Bigabs{\log Z_X(\theta) } \geq A \sqrt{\log N} }} \leq c_1 e^{- \frac{A}{2} \left(A \wedge \frac{\sqrt{\log N}}{2} \right) }
\end{eqnarray*}

\end{lemma}

\begin{proof}
 For all $\lambda \geq 0$,
\vspace{-0.3cm}

\begin{eqnarray*}
%
%
%
\intpi{\Proba{SU(N)}{ \Bigabs{\log Z_X(\theta) } \geq A \sqrt{\log N} }} & \leq & e^{ - \lambda A \sqrt{\log N} } \intpi{ \Espr{SU(N)}{  e^{\lambda   \abs{ \log Z_X(\theta) }  } } } \\
%
 				& \leq & e^{ - \lambda A \sqrt{\log N} }  \Espr{U(N)}{  e^{\lambda   \abs{ \log Z_X(0) }  } }  ~ \mbox{(by \eqref{Disintegration})} \\
 				& \leq & e^{ - \lambda A \sqrt{\log N} }  \Espr{U(N)}{  e^{\lambda   \prth{ \abs{ \Re\log Z_X(0)} + \abs{\Im\log Z_X(0) } }  } }  \\
\end{eqnarray*}

Using the inequality $ e^{\abs{a} + \abs{b}} \leq e^{a + b} + e^{a - b} + e^{-a + b} + e^{-a - b} $, valid for all $ a, b \in \Rr $, and writing the right-hand
side of this inequality as $ 4 \Esp{e^{Ba + B'b}} $ for $ B $ and $ B' $ being two independent Bernoulli random variables independent of $ U $ such that $ \Prob{B = 1} = 1 \! - \Prob{B = \! -1} = 1/2 $, we have:

\vspace{-0.3cm}
\begin{multline*}
\intpi{\Proba{SU(N)}{ \Bigabs{\log Z_X(\theta) } \geq A \sqrt{\log N} }} \\
\leq 4 e^{ - \lambda A \sqrt{\log N} }  \Espr{U(N)}{  e^{\lambda   \prth{  B \, \Re\log Z_X(0) + B' \Im\log Z_X(0) \,  }  } }.
\end{multline*}

We now use the fact (\cite{KS} and \cite{H}, p.16) that for $s, t \in \Cc$ such that $\Re(s+it)$ and $\Re(s-it)$ are strictly larger than $-1$:

\vspace{-0.3cm}
\begin{equation}\label{JointFourier}
\Espr{U(N)}{  e^{  s \, \Re\log Z_X(0) + t \, \Im\log Z_X(0)   } } \\
= \frac{ G\prth{ 1 + \frac{s + it}{2} } G\prth{ 1 + \frac{s - it}{2} } G\prth{ 1 + N } G\prth{ 1 + N + s } }{G\prth{ 1 + N + \frac{s + it}{2} } G\prth{ 1 + N + \frac{s - it}{2} } G\prth{ 1 + s } } 							
\end{equation}

\noindent where $G$ is the Barnes $ G $-function, defined for all $z \in \Cc$, by
$$G(z+1) := (2 \pi)^{z/2} e^{-[(1+ \gamma) z^2 + z]/2} \prod_{n=1}^{\infty} \left( 1 + \frac{z}{n} \right)^n e^{-z + (z^2/2n)},$$
$\gamma$ being the Euler constant.  The function $G$ also satisfies the functional equation $G(z+1) = \Gamma(z) G(z)$.

In other words, one has
$$ \Espr{U(N)}{  e^{  s \, \Re\log Z_X(0) + t \, \Im\log Z_X(0)   } } =   \frac{ G\prth{ 1 + \frac{s + it}{2} } G\prth{ 1 + \frac{s - it}{2} } }{  G\prth{ 1 + s } } N^{ \prth{s^2 + t^2}/4} G_{N,s,t},$$
where, by the classical estimates of the Barnes function,
$$G_{N,s,t} := N^{ - \prth{s^2 + t^2}/4} \frac{  G\prth{ 1 + N } G\prth{ 1 + N + s } }{G\prth{ 1 + N + \frac{s + it}{2} } G\prth{ 1 + N + \frac{s - it}{2} } }$$
tends to $1$ when $N$ goes to infinity, uniformly on $s$ and $t$ if these parameters are bounded.

For any sequence $(\lambda_N)_{N \geq 1}$ such that $\lambda_N \in [0,1/2]$, one has (taking $s = \lambda_N B$ and $t = \lambda_N B'$):
$$ \Espr{U(N)}{  e^{\lambda_N  \prth{  B \, \Re\log Z_X(0) + B' \Im\log Z_X(0) \,  }  } }  = M(\lambda_N) N^{  \frac{\lambda_N^2}{2} },$$
where $$ M(\lambda_N) := \Espr{}{ \frac{ G\prth{1 + \lambda_N \frac{B + i B'}{2} } G\prth{1 + \lambda_N \frac{B - i B'}{2} } }{ G(1 + \lambda_N B )} G_{N,\lambda_N B, \lambda_N B' }}.$$
Since the function $G$ is holomorphic, with no zero on the half-plane $\{\Re > 0\}$, and since $G_{N,\lambda B, \lambda B' }$ tends to $1$ when $N$ goes to infinity, uniformly on $\lambda
\in [0,1/2]$, the quantity $M(\lambda)$ is uniformly bounded by
some universal constant $c' > 0$, for $\lambda \in [0,1/2]$. Hence,
$$  \Espr{U(N)}{  e^{\lambda_N  \prth{  B \, \Re\log Z_X(0) + B' \Im\log Z_X(0) \,  }  } }  \leq c' N^{  \frac{\lambda_N^2}{2} },$$
for $N$ going to infinity, which implies:

\vspace{-0.3cm}
$$
\intpi{\Proba{SU(N)}{ \Bigabs{\log Z_X(\theta) } \geq A \sqrt{\log N} }} \leq 4 c' e^{ - \lambda_N A \sqrt{\log N} + (\lambda_N^2 \log N) /2}.
$$
Now, taking $\lambda_N = (1/2) \wedge (A/\sqrt{\log N})$ gives
\begin{align*}
\intpi{\Proba{SU(N)}{ \Bigabs{\log Z_X(\theta) } \geq A \sqrt{\log N} }}
 & \leq 4 c' e^{ - \lambda_N \sqrt{\log N} [ A - (\lambda_N \sqrt{\log N})/2]}
\\ & \leq 4 c' e^{ - \lambda_N \sqrt{\log N} [ A - (A/\sqrt{\log N})(\sqrt{\log N})/2]}
\\ & \leq 4 c' e^{ - \lambda_N \sqrt{\log N} (A/2)}
\\ & \leq 4 c' e^{ - [ (\sqrt{\log N}/2) \wedge A] (A/2)}.
\end{align*}

\end{proof}
\subsection{An estimate on the imaginary part of the log-characteristic polynomial}
From the previous result, we  obtain the following estimate for the imaginary part of the log-characteristic polynomial:
\begin{lemma} \label{estimateimaginarypart} There exists a universal constant $ c'_1 > 0 $ such that for all $N \geq 2$, $A \geq 0 $,
 and  $\theta \in \Rr$,
\vspace{-0.3cm}
\begin{eqnarray*}
\Proba{SU(N)}{ \Bigabs{\Im \log Z_X(\theta) } \geq A \sqrt{\log N} } \leq c'_1 e^{- \frac{A}{2} \left(A \wedge \frac{\sqrt{\log N}}{2} \right) }
\end{eqnarray*}

\end{lemma}
\begin{proof}
 We use here the probabilistic splitting established in \cite{BHNY} which shows that   (see also \cite{BNN} for an infinite-dimensional point of view), for any $U \in U(N)$, there exists, for $1 \leq j \leq N$, $x_j$ on the unit sphere of
$\Cc^j$, uniquely determined, such that
\begin{equation}\label{splittingreflections}
U = R(x_N) \,  \left( \begin{array}{cc}
R(x_{N-1}) & 0  \\
0 & 1  \end{array} \right) \left( \begin{array}{cc} R(x_{N-2}) & 0  \\ 0 & I_2 \end{array} \right)
\cdots \left( \begin{array}{cc} R(x_{1}) & 0  \\ 0 & I_{N-1} \end{array} \right),
\end{equation}
where $R(x_j)$ denotes the unique unitary matrix in $U(j)$ sending the last basis vector $e_j$ of $\mathbb{C}^j$ to $x_j$, and such that the image
of $I_j - R(x_j)$ is the vector space generated by $e_j - x_j$.

Moreover, the characteristic polynomial of $U(N)$ is given by
$$Z_U (0) = \prod_{j = 1}^{N} (1 - \langle x_j, e_j \rangle),$$
and then its logarithm is
\begin{equation}
\log Z_U (0) = \sum_{j = 1}^{N} \log(1 - \langle x_j, e_j \rangle), \label{logarithmBHNY}
\end{equation}
when $1$ is not an eigenvalue of $U$, taking the principal branch of the logarithm on the right-hand side.
Notice that the determination of the logarithm given by this formula fits with the definition involving the eigenangles \eqref{definitionlogcharacteristic}. Indeed,
the two formulas depend continuously on the matrix $U$, on the connected set $\{U \in U(N), 1 \notin \operatorname{Spec}(U)\}$, and their exponentials are equal, hence,
 it is sufficient to check that they coincide for one matrix $U$, for example $-I_N$ (in this case, $x_j = -e_j$ for all $j$ and the two formulas give $N \log 2$).

If $U$ follows the uniform distribution on $U(N)$, then the vectors $(x_j)_{1 \leq j \leq N}$ are independent and $x_j$ is uniform on the sphere of $\Cc^j$.
The determinant of $U$ is equal to the product of the determinants of $R(x_j)$ for $1 \leq j \leq N$, and since $R(x_1)$ is the multiplication by $x_1$ on $\Cc$, one
has $$\det(U) = x_1 \prod_{j=2}^{N} \Gamma_j (x_j),$$
where $\Gamma_j$ is a function from $\Cc^j$ to the unit circle $\Uu$. From this, let us deduce that under the measure $\Pp_{SU(N), \theta}$:
\begin{enumerate}
 \item The vectors $(x_j)_{2 \leq j \leq N}$ are independent, $x_j$ being uniform on the unit sphere of $\Cc^j$.
\item The value of $x_1 \in \Uu$ is uniquely determined by the determinant $\det(U) = e^{i N \theta}$,
$$ x_1 = e^{i N \theta} \prod_{j=2}^{N} [\Gamma_j (x_j)]^{-1}.$$
\end{enumerate}
Indeed, let $\Pp'_{SU(N), \theta}$ be the probability measure on the image of $SU(N)$ by the multiplication by $e^{i \theta}$, under which the law of
$(x_j)_{1 \leq j \leq N}$ is given by the two items above. This probability measure can be constructed as the law of
the random matrix $U$ given by the formula \eqref{splittingreflections}, where $(x_j)_{1 \leq j \leq N}$
are random vectors whose joint distribution is given by the items (1) and (2) just above.
 We now have to prove that $\Pp_{SU(N), \theta} = \Pp'_{SU(N), \theta} $.
Let us first notice that the joint law of $(x_j)_{2 \leq j \leq N}$, under the probability measure $\Pp'_{SU(N), \theta}$, does not depend on $\theta$. Hence, under the averaged measure
$$\int_0^{2 \pi} \Pp'_{SU(N), \theta} \, \frac{d\theta}{2 \pi},$$
the vectors $(x_j)_{2 \leq j \leq N}$ still have the same law, i.e. they are independent and $x_j$ is uniform
 on the unit sphere of $\Cc^j$. Moreover, conditionally on $(x_j)_{2 \leq j \leq N}$,
$x_1 = e^{i N \theta} \prod_{j=2}^{N} [\Gamma_j (x_j)]^{-1}$, where $\theta$ is uniform on $[0, 2\pi)$. Hence, $(x_j)_{1 \leq j \leq N}$ are independent,
$x_1$ is uniform on $\Uu$, and then $x_j$ in uniform on the unit sphere of $\Cc^j$ for all $j \in \{1, \dots, N\}$, which implies
$$\int_0^{2 \pi} \Pp'_{SU(N), \theta} \, \frac{d\theta}{2 \pi} =  \Pp_{U(N)} = \int_0^{2 \pi} \Pp_{SU(N), \theta} \, \frac{d\theta}{2 \pi}.$$
Now, $\Pp_{SU(N), 2 \pi/N }$ is the image of $\Pp_{SU(N)}$ by multiplication by $e^{i 2 \pi/N} I_N$, which is a matrix in $SU(N)$: the invariance
property defining the Haar measure $\Pp_{SU(N)}$ implies that $\Pp_{SU(N), 2 \pi/N } = \Pp_{SU(N)}$, and then
$\theta \mapsto \Pp_{SU(N), \theta}$ is $(2 \pi/N)$-periodic. It is the same for $\theta \mapsto \Pp'_{SU(N), \theta}$, since
the values of $x_1, \dots x_N$ involved in the definition of $\Pp'_{SU(N), \theta}$ do not change if we add a multiple of $2 \pi /N$ to $\theta$. Hence,
$$\int_0^{2 \pi/N} \Pp'_{SU(N), \theta} \, \frac{N d\theta}{2 \pi} = \int_0^{2 \pi/N} \Pp_{SU(N), \theta} \, \frac{N d\theta}{2 \pi}.$$
Now, let $F$ be a continuous, bounded function from $U(N)$ to $\mathbb{R}$. By applying the equality above to the function
$U \mapsto F(U) \Unens{\det{U} \in \{e^{iN \theta}, \, \theta \in I\} }$, for an interval $I \subset [0, 2\pi/N)$, one deduces with obvious notation that:
$$\int_I \Ee'_{SU(N), \theta} (F) \frac{d \theta}{|I|} = \int_I \Ee_{SU(N), \theta} (F) \frac{d \theta}{|I|},$$
where $|I|$ is the length of $I$.
Now, by definition of $\Pp_{SU(N), \theta}$ and $\Pp'_{SU(N), \theta}$, the first measure is the image of $\Pp_{SU(N)}$ by multiplication
by $e^{i \theta}$, and the second measure is the image of $\Pp'_{SU(N),0}$ by right multiplication by the matrix
$\left( \begin{array}{cc}  e^{i N \theta} & 0  \\ 0 & I_{N-1} \end{array} \right)$. Hence, by continuity and boundedness of $F$, and by dominated convergence,
 $\Ee_{SU(N), \theta} (F)$ and $\Ee'_{SU(N), \theta} (F)$ are continuous with respect to $\theta$. By considering a sequence $(I_r)_{r \geq 1}$ of
intervals containing a given value of $\theta$ and whose length tends to zero, one deduces, by letting $r \rightarrow \infty$,
$$\Ee'_{SU(N), \theta} (F) =  \Ee_{SU(N), \theta} (F).$$
We now get the equality $\Pp_{SU(N), \theta} = \Pp'_{SU(N), \theta} $, and then the law of $(x_j)_{1 \leq j \leq N}$ under
$\Pp_{SU(N), \theta} $ described above.

 Hence, the sequence $(x_j)_{2 \leq j \leq N}$ has the same law under $\Pp_{SU(N), \theta} $ and
$\Pp_{U(N)}$. We now use this fact to construct a coupling between these two probability measures on the unitary group.

The general principle of coupling is the following: when we want to show that two probability distributions
$\Pp_1$ and $\Pp_2$ on a metric space have a similar behavior, a possible strategy is to construct a couple $(U,U')$
of random variables defined on the same probability space endowed with a probability $ \Pp $, such that the law of $U$ under $ \Pp $ is $\Pp_1$, the law of
$U'$ under $ \Pp $ is $\Pp_2$, and the distance between $U$ and $U'$ is small with high probability.
In the present situation, we take $(x'_j)_{1 \leq j \leq N}$ independent, $x'_j$ uniform
on the unit sphere of $\Cc^j$ for all $j \in \{1, \dots, N\}$, and we construct, by using \eqref{splittingreflections},
a random matrix $U'$ following $\Pp_{U(N)}$. Then, we do the coupling by taking $x_j := x'_j$ for
$2 \leq j \leq N$ and
 $$ x_1 := e^{i N \theta} \prod_{j=2}^{N} [\Gamma_j (x_j)]^{-1},$$
 which gives a random matrix $U$ following $\Pp_{SU(N), \theta}$.
From the fact that $x_j = x'_j$ for $j \geq 2$ and the equation \eqref{logarithmBHNY}, we get the following:
$$\log Z_U (0) - \log Z_{U'}(0) = \log(1 - x_1) - \log(1 - x'_1),$$
and in particular,
$$\left| \Im \log Z_U (0) - \Im \log Z_{U'} (0) \right| \leq \pi.$$
Now, for $B := \left(A - \frac{\pi}{\sqrt{\log N}}\right)_+$, one gets:
\begin{align*}
 \Proba{SU(N)}{ \Bigabs{\Im \log Z_X(-\theta) } \geq A \sqrt{\log N} } & = \Proba{SU(N), \theta}{ \Bigabs{\Im \log Z_X(0) } \geq A \sqrt{\log N} } \\
& = \Proba{}{ \Bigabs{\Im \log Z_U(0) } \geq A \sqrt{\log N} }
\\ & \leq \Proba{}{ \Bigabs{\Im \log Z_{U'}(0) } \geq A \sqrt{\log N} - \pi}
\\ & = \Proba{U(N)}{ \Bigabs{\Im \log Z_{X}(0) } \geq B \sqrt{\log N}}
\\ & = \intpi{\Proba{SU(N)}{ \Bigabs{ \Im \log Z_X(\theta) } \geq B \sqrt{\log N} }} 
\\ & \leq \intpi{\Proba{SU(N)}{ \Bigabs{ \log Z_X(\theta) } \geq B \sqrt{\log N} }} 
\\ & \leq c_1 e^{- \frac{B}{2} \left(B \wedge \frac{\sqrt{\log N}}{2} \right) }
\end{align*}
Now, if $B \leq \frac{\sqrt{\log N}}{2}$,
\begin{align*}
\frac{A}{2} \left( A \wedge \frac{\sqrt{\log N}}{2} \right) \leq \frac{A^2}{2} & \leq \frac{1}{2} \left( B + \frac{\pi}{\sqrt{\log N}}\right)^2 =
\frac{B^2}{2} + \frac{B\pi}{\sqrt{\log N}} + \frac{\pi^2}{2 \log N}
\\ & \leq \frac{B^2}{2} + \frac{\pi}{2} + \frac{\pi^2}{2 \log 2}
\\ & = \frac{B}{2} \left( B \wedge \frac{\sqrt{\log N}}{2} \right) + \frac{\pi}{2} + \frac{\pi^2}{2 \log 2}.
\end{align*}
If $B \geq \frac{\sqrt{\log N}}{2}$,
\begin{align*}
\frac{A}{2} \left( A \wedge \frac{\sqrt{\log N}}{2} \right) \leq \frac{A\sqrt{\log N}}{4} & \leq \frac{\sqrt{\log N}}{4} \left( B + \frac{\pi}{\sqrt{\log N}}\right) = 
\frac{B \sqrt{\log N}}{4} + \frac{\pi}{4}
\\ & = \frac{B}{2} \left( B \wedge \frac{\sqrt{\log N}}{2} \right) + \frac{\pi}{4}.
\end{align*}
Hence, we get Lemma \ref{estimateimaginarypart}, with
$$c'_1 = c_1 e^{\frac{\pi}{2} + \frac{\pi^2}{2 \log 2}}.$$

\end{proof}

\subsection{Bound on the concentration of the law of the log-characteristic polynomial}
\begin{lemma} \label{boundconcentration}
For $N \geq 4$, $\theta \in [0, 2\pi)$, $x_0 \in \mathbb{R}$ and  $\delta \in (0,1/2)$, one has
$$\mathbb{P}_{SU(N)} [|\log |Z_X(\theta)| - x_0| \leq \delta \sqrt{\log N}] \leq C \delta \log( 1 / \delta),$$
where $C > 0$ is a universal constant.
\end{lemma}
\begin{proof}
 The proof of Lemma \ref{boundconcentration} needs several steps.

  \begin{sublemma}
For $j \geq 1$ integer, $s, t \in \mathbb{R}$, let us define
$$Q(j,s,t) := \frac{\left(j + \frac{it -s}{2}\right) \left(j + \frac{it+s}{2} \right)}{j(j+it)}.$$
Then,
\begin{enumerate}
\item For $s^2 + t^2 \geq 8 j^2$, $|Q(j,s,t)| \geq \max\left(1, \frac{\sqrt{s^2+t^2}}{8j} \right)$.
\item For $j^2 \leq s^2 + t^2 \leq 8j^2$, $|Q(j,s,t)| \leq 1$.
\item For $ s^2 + t^2 \leq j^2$, $|Q(j,s,t)| \leq e^{-(s^2+t^2)/10j^2}$.
\end{enumerate}
\end{sublemma}


\begin{proof}
One has:
\begin{equation}
Q(j,s,t) = \frac{1 - \frac{s^2+t^2}{4j^2} + it/j}{ 1 + it/j}. \label{Q1}
\end{equation}
If $s^2 + t^2 \leq 8j^2$, it is immediate that the numerator has a smaller
absolute value than the denominator, i.e.  $|Q(j,s,t)| \leq 1$.
Moreover,
$$ |Q(j,s,t)|^2  = \frac{ 1 - \frac{s^2+t^2}{2j^2} + \frac{(s^2+t^2)^2}{16j^4} + \frac{t^2}{j^2}} { 1 + \frac{t^2}{j^2}}
 = 1 - \frac{\left(\frac{s^2 + t^2}{2j^2}\right) \left( 1 - \frac{s^2 + t^2}{8j^2} \right) }{ 1 + \frac{t^2}{j^2}} $$
and in the case where $ s^2 + t^2 \leq j^2$, one deduces
$$ |Q(j,s,t)|^2  \leq 1 - \frac{7(s^2+ t^2)}{32 j^2} $$
and then
$$|Q(j,s,t)| \leq e^{-7(s^2+t^2)/64j^2} \leq  e^{-(s^2+t^2)/10j^2}.$$
Now, if $s^2 + t^2 \geq 8j^2$, the numerator in \eqref{Q1} has a larger absolute value than the denominator, and
then $|Q(j,s,t)| \geq 1$. Moreover, since $(s^2 + t^2)/8j^2 \geq 1$,
\begin{align*}
|Q(j,s,t)|^2 &  = \frac{ \left(\frac{s^2+t^2}{4j^2} - 1\right)^2 + \frac{t^2}{j^2}} { 1 + \frac{t^2}{j^2}}
\geq \frac{ \left(\frac{s^2+t^2}{8j^2} \right)^2 + \frac{t^2}{j^2}} { 1 + \frac{t^2}{j^2}}
 \geq  \frac{ \left(\frac{s^2+t^2}{8j^2} \right)^2 + \frac{s^2 + t^2}{j^2}} { 1 + \frac{s^2 + t^2}{j^2}}
 \\ & \geq \frac{1}{64} \, . \frac{ \left(\frac{s^2+t^2}{j^2} \right)^2 + \frac{s^2 + t^2}{ j^2}} { 1 + \frac{s^2 + t^2}{j^2}}
 = \frac{s^2 + t^2}{64j^2},
\end{align*}
which finishes the proof of the sublemma.
\end{proof}
\begin{sublemma}
Let $j \geq 1$ be an integer, let $\rho_j$ and $\sigma_j$ be the real and imaginary parts of $\log (1 - \sqrt{\beta_{1,j-1}} e^{i \theta})$,
where $\beta_{1,j-1}$ is a beta random variable with
$\beta(1,j-1)$ distribution and $\theta$ is independent of $\beta_{1,j-1}$, uniform on $[0,2 \pi]$.
Then, for $s, t \in \mathbb{R}$,
$$|\mathbb{E} [e^{i (t \rho_j + s \sigma_j)} ]| \leq e^{-(s^2 + t^2)/30j}$$
if $s^2 + t^2 \leq 8j^2$,
and
$$|\mathbb{E} [e^{i (t \rho_j + s \sigma_j)} ]| \leq \frac{8}{\sqrt{s^2+ t^2}}$$
if $s^2 + t^2 \geq 8j^2$ and $j \geq 2$.
\end{sublemma}
\begin{proof}
For $t \in \mathbb{R}$ and $s \in \mathbb{C}$ with
real part strictly between $-1$ and $1$,
\begin{equation}
\mathbb{E} [e^{i (t \rho_j + s \sigma_j)} ]  = \frac{\Gamma(j) \Gamma(j+it)}{\Gamma \left(j + \frac{it -s}{2}\right)\Gamma \left(j + \frac{it +s}{2}\right)} \label{a}
\end{equation}
(see \cite{BHNY}).
Now, if $t$ is fixed, the function
$$s \mapsto \mathbb{E} [e^{i (t \rho_j + s \sigma_j)} ]$$
is holomorphic, since the imaginary part is uniformly bounded (by $\pi/2$), which implies that \eqref{a} holds for all $t \in \mathbb{R}$, $s \in \mathbb{C}$, and in
particular for all $s,t \in \mathbb{R}$.
Moreover,
$$ \frac{\Gamma( k) \Gamma(k + it)}{\Gamma \left(k + \frac{it -s}{2}\right)\Gamma \left(k + \frac{it +s}{2}\right)}
\underset{k \rightarrow \infty}{\longrightarrow} 1,$$
since $\Gamma(k + z)/ \Gamma(k)$ is equivalent to $k^z$ for all $z \in \mathbb{C}$. 
Hence, by using the equation $\Gamma(z+ 1) = z \Gamma(z)$, one deduces:
$$\mathbb{E} [e^{i (t \rho_j + s \sigma_j)} ]  = \prod_{k = j}^{\infty} \frac{\left(k+ \frac{it -s}{2}\right)\left(k + \frac{it +s}{2}\right)}{ k( k +it)}
= \prod_{k = j}^{\infty} Q(k,s,t).$$
If $s^2 + t^2 \leq 8j^2$, then $|Q(k,s,t)| \leq 1$ for all $k \geq j$ and
$|Q(k,s,t)| \leq e^{-(s^2+t^2)/10k^2}$ for all $k \geq 3j$. Hence
$$|\mathbb{E} [e^{i (t \rho_j + s \sigma_j)} ] | \leq \prod_{k=3j}^{\infty} e^{-(s^2+t^2)/10k^2} \leq \prod_{k=3j}^{\infty} e^{-(s^2+t^2)/10k(k+1)} =
e^{-(s^2 + t^2)/30j}.$$
\noindent
Now let us assume $s^2 + t^2 \geq 8j^2$. One has:
$$ \mathbb{E} [e^{i (t \rho_j + s \sigma_j)} ] =  \frac{\Gamma(1) \Gamma(1+it)}{\Gamma \left(1 + \frac{it -s}{2}\right)\Gamma \left(1 + \frac{it +s}{2}\right)}
\, \prod_{k = 1}^{j-1} \frac{1}{ Q(k,s,t)}$$
where all the factors $\frac{1}{Q(k,s,t)}$ have absolute value bounded by one. By considering the case where $j=1$, one
deduces
$$  \left| \frac{\Gamma(1) \Gamma(1+it)}{\Gamma \left(1 + \frac{it -s}{2}\right)\Gamma \left(1 + \frac{it +s}{2}\right)} \right| \leq 1,$$
and then, for $j \geq 2$,
$$|\mathbb{E} [e^{i (t \rho_j + s \sigma_j)} ] | \leq \frac{1}{|Q(1,s,t)|} \leq \frac{8}{\sqrt{s^2+t^2}}.$$
\end{proof}
\begin{sublemma} \label{density}
For $N \geq 4$ and $\theta \in [0,2 \pi)$, the distribution of $\log (Z_X(\theta))$ under Haar measure on $U(N)$ has a density with respect to
Lebesgue measure on $\mathbb{C}$, which is continuous and bounded by $C_0/\log(N)$, where $C_0 > 0$ is a universal constant.
\end{sublemma}
\begin{proof}
By the results in \cite{BHNY} and the previous sublemma, one checks that the characteristic function $\Phi$ of $\log (Z_X(\theta)) \in \mathbb{C} \sim \mathbb{R}^2$  is given
by $$\Phi (s,t) = \prod_{j=1}^N \mathbb{E} [e^{i (t \rho_j + s \sigma_j)} ].$$
If $s^2 + t^2 \geq 32N$, one has $s^2 + t^2 \geq 128 \geq 8j^2$ for $j \in \{2,3,4\}$. Hence,
$$|\Phi(s,t)| \leq |\mathbb{E} [e^{i (t \rho_2+ s \sigma_2)} ]| |\mathbb{E} [e^{i (t \rho_3 + s \sigma_3)} ]||\mathbb{E} [e^{i (t \rho_4 + s \sigma_4)} ]|
\leq \frac{512}{(s^2+t^2)^{3/2}}.$$
If $s^2 + t^2 \leq 32N$, then $s^2 + t^2 \leq 8j^2$ for all $j \geq 2 \sqrt{N}$. Hence,
$$|\Phi(s,t)| \leq \prod_{2 \sqrt N \leq j \leq N}   \mathbb{E} [e^{i (t \rho_j + s \sigma_j)} ]
\leq \exp \left( - (s^2 + t^2) \sum_{2 \sqrt{N} \leq j \leq N} \frac{1}{30j} \right).$$
Since $e^{1/j} \geq \frac{j+1}{j}$, one deduces
\begin{align*}
|\Phi(s,t)| &   \leq \prod_{2 \sqrt{N} \leq j \leq N} \left(\frac{j}{j+1} \right)^{(s^2+t^2)/30} \leq \left(\frac{2\sqrt{N}+1}{N+1}\right)^{(s^2+t^2)/30}
\\ & \leq \left(\frac{3\sqrt{N}}{N}\right)^{(s^2+t^2)/30}  = e^{- \log(N/9)(s^2 + t^2)/60}.
\end{align*}
Now, for $N \geq 10$,
\begin{align*}
\int_{\mathbb{R}^2}  |\Phi(s,t)| ds dt & \leq \int_{\mathbb{R}^2}  \frac{512}{(s^2+t^2)^{3/2}} \, \Unens{s^2+t^2 \geq 32N} ds dt +
 \int_{\mathbb{R}^2} e^{- \log(N/9)(s^2 + t^2)/60} \, \Unens{s^2+t^2 \leq 32N} ds dt 
 \\ & = \pi \left( \int_{0}^{32N} e^{- u \log(N/9) / 60} du + \int_{32N}^{\infty} \frac{512}{u^{-3/2}} du \right)
 \\ & \leq \frac{60 \pi }{\log (N/9)} + 1024 \pi (32N)^{-1/2} \leq \frac{10000}{\log N},
 \end{align*}
\noindent
 and for $N \in \{4,5,6,7,8,9\}$,
 \begin{align*}
 \int_{\mathbb{R}^2}  |\Phi(s,t)| ds dt & \leq \int_{\mathbb{R}^2}  \frac{512}{(s^2+t^2)^{3/2}} \, \Unens{s^2+t^2 \geq 32N} ds dt +
 \int_{\mathbb{R}^2} \, \Unens{s^2+t^2 \leq 32N} ds dt
 \\ & = \pi \left( \int_{0}^{32N} du + \int_{32N}^{\infty} \frac{512}{u^{-3/2}} du \right)
 \\ & \leq  32 \pi N + 1024 \pi (32N)^{-1/2} \leq 288 \pi + 1024 \pi (128)^{-1/2} \leq \frac{10000}{\log 9}.
 \end{align*}
By applying Fourier inversion, we obtain Sublemma \ref{density}.
\end{proof}

Let us now go back to the proof of Lemma \ref{boundconcentration}. For any $X \in U(N)$ with eigenvalues $(e^{i\theta_j})_{1 \leq j \leq N}$, one has, in
the case where $e^{i \theta} \neq e^{i \theta_j}$ for all
$j \in \{1, \dots, N\}$, and modulo $\pi$,
\begin{align*}
\mathcal{I}:= \Im ( \log (Z_X(\theta)) &
=  \sum_{1 \leq j \leq N} \Im (\log (1- e^{i(\theta_j-\theta)}))
\\ & = \frac{1}{2} \sum_{1 \leq j \leq N} (\theta_j -\theta) +  \sum_{1 \leq j \leq N} \Im (\log (e^{-i(\theta_j-\theta)/2} - e^{i(\theta_j-\theta)/2}))
\\ & = \frac{1}{2} \Im(\log \, \det (X)) - \frac{N \theta}{2} + \sum_{1 \leq j \leq N} \Im \left(\log\left(-2 i \sin \left(\theta_j - \theta)/2 \right)\right) \right)
\\  & = \frac{\mathcal{J}}{2} - \frac{N (\theta + \pi)}{2},
\end{align*}
where $\mathcal{J}$ denotes the version of $\Im(\log \, \det (X))$ lying on the interval $(-\pi, \pi]$. Hence, for any $\epsilon \in (0, \pi)$,
$|\mathcal{J}| \leq \epsilon$ if and only if $\mathcal{I}$ is on an interval of the form $\left[\frac{2k \pi - \epsilon - N (\theta + \pi)}{2}, \frac{2k \pi + \epsilon - N (\theta + \pi)}{2} \right]$ for
some $k \in \mathbb{Z}$. Now, for some $A > 0$ chosen later in function of $\delta$, let $\Phi$ be a continuous function from $\mathbb{C}$ to $[0,1]$ such that $\Phi(z) = 1$ if $|\Re z - x_0|
 \leq \delta  \sqrt{\log N}$ and  $|\Im z| \leq A \sqrt{\log N}$, and such that $\Phi(z) = 0$ for $|\Re z - x_0|
 \geq 2 \delta  \sqrt{\log N}$ or  $|\Im z| \geq 2 A \sqrt{\log N}$.
 For $\epsilon \in (0, \pi)$, and under the Haar measure $\mathbb{P}_{U(N)}$ on $U(N)$,
\begin{multline*}
\frac{\pi}{\epsilon} \mathbb{E}_{U(N)} \left[\Phi( \log (Z_X(\theta))) \Unens{|\mathcal{J}| \leq \epsilon}\right]\\
 = \frac{\pi}{\epsilon} \sum_{k \in \mathbb{Z}} \mathbb{E}_{U(N)} \left[ \Phi( \log (Z_X(\theta))) \Unens{ \frac{2k \pi - \epsilon - N (\theta + \pi)}{2} \leq
\mathcal{I} \leq \frac{2k \pi + \epsilon - N (\theta + \pi)}{2}} \right] \\
 = \frac{\pi}{\epsilon} \sum_{k \in \mathbb{Z}}
\int_{-\infty}^{\infty} dx  \int_{(2k \pi - \epsilon - N (\theta + \pi))/2}^{(2k \pi + \epsilon - N (\theta + \pi))/2}  dy \, D(x + iy) \Phi(x + iy)
\end{multline*}
$$ =  \pi  \sum_{k \in \mathbb{Z}} \int_{-\infty}^{\infty} dx \int_{-1/2}^{1/2} du \, D(x + i[ k \pi -  N (\theta + \pi))/2 + u \epsilon]) \Phi(x + i[ k \pi -  N (\theta + \pi))/2 + u \epsilon]),
$$
where $D$ denotes the density of the law of $\log (Z_X(\theta))$, with respect to the Lebesgue measure. Now,
$$ D(x + i[ k \pi -  N (\theta + \pi))/2 + u \epsilon]) \Phi(x + i[ k \pi -  N (\theta + \pi))/2 + u \epsilon])$$ is uniformly bounded by the overall maximum of $D$ and
vanishes as soon as $|x - x_0| \geq 2 \delta \sqrt{\log N}$ or $|k| \pi \geq N (|\theta| + \pi)/2 + \pi/2 +  2 A \sqrt{\log N}$. Since $D$ and $\Phi$ are continuous functions, one can apply dominated convergence and deduce that
$$\frac{\pi}{\epsilon} \mathbb{E}_{U(N)} \left[\Phi( \log (Z_X(\theta))) \Unens{|\mathcal{J}| \leq \epsilon}\right]$$
converges to
$$ \pi  \sum_{k \in \mathbb{Z}} \int_{-\infty}^{\infty} \, D(x + i[ k \pi -  N (\theta + \pi))/2]) \Phi(x + i[ k \pi -  N (\theta + \pi))/2]) dx$$
when $\epsilon$ goes to zero. On the other hand, if the matrix $X$ follows $\Pp_{SU(N)}$ and if $T$ is an independent uniform variable on
$(-\pi, \pi]$, then $X e^{i T/N}$ follows $\Pp_{U(N)}$ and its determinant is $e^{i T}$. One deduces:
\begin{align*}
\frac{\pi}{\epsilon} \mathbb{E}_{U(N)} \left[\Phi( \log (Z_X(\theta))) \Unens{|\mathcal{J}| \leq \epsilon}\right]
& = \frac{\pi}{\epsilon} \mathbb{E}_{SU(N)} \left[\Phi( \log (Z_{Xe^{iT/N}}(\theta))) \Unens{|T| \leq \epsilon}\right]
\\ & = \frac{1}{2 \epsilon} \int_{- \epsilon}^{\epsilon} \mathbb{E}_{SU(N)} \left[\Phi( \log (Z_{Xe^{it/N}}(\theta)))\right] \, dt
\\ & = \int_{-1/2}^{1/2}  \mathbb{E}_{SU(N)} \left[\Phi( \log (Z_{Xe^{2i v \epsilon/N}}(\theta)))\right] \, dv
\end{align*}
\noindent
Now, the function $X \mapsto \Phi(\log(Z_X(\theta)))$ is continuous from $U(N)$ to $[0,1]$, since $\Phi$ is continuous with compact support and $X \mapsto \log(Z_X(\theta)))$ has
 discontinuities only at points where its real part goes to $-\infty$. One can then apply dominated convergence and obtain:
 $$\frac{\pi}{\epsilon} \mathbb{E}_{U(N)} \left[\Phi( \log (Z_X(\theta))) \Unens{|\mathcal{J}| \leq \epsilon}\right]
\underset{\epsilon \rightarrow 0}{\longrightarrow}  \mathbb{E}_{SU(N)} \left[\Phi( \log (Z_{X}(\theta)))\right].$$
By comparing to the convergence obtained just above, one deduces
 $$\mathbb{E}_{SU(N)} \left[\Phi( \log (Z_{X}(\theta)))\right] =
 \pi  \sum_{k \in \mathbb{Z}} \int_{-\infty}^{\infty} \, D(x + i[ k \pi -  N (\theta + \pi))/2]) \Phi(x + i[ k \pi -  N (\theta + \pi))/2]) dx.$$
 Since $D(z) \leq C_0/\log N$ and $$\Unens{|x-x_0| \leq \delta \sqrt{\log N}, |y| \leq A \sqrt{\log N}} \leq \Phi(x + iy) \leq
\Unens{|x-x_0| \leq 2 \delta  \sqrt{\log N}, |y| \leq 2 A \sqrt{\log N}}$$ for all $x, y \in \mathbb{R}$, one deduces
$$\mathbb{P}_{SU(N)} [|\log |Z_X(\theta)| - x_0| \leq \delta \sqrt{\log N}, |\Im \log Z_X(\theta)| \leq A \sqrt{\log N}] \leq \frac{\pi d L C_0}{\log N},$$
where $ d = 4 \delta \sqrt{\log N} $ is the length of the interval
$[x_0 - 2 \delta \sqrt{\log N}, x_0 + 2 \delta \sqrt{\log N}]$
and $L$ is the number of integers $k$ such that $|k \pi -  N (\theta + \pi))/2| \leq 2 A \sqrt{\log N}$.
Now, it is easy to check that $L \leq 1+ \frac{4 A \sqrt{\log N}}{\pi} $, and
then
$$\mathbb{P}_{SU(N)} \left[|\log |Z_X(\theta)| - x_0| \leq \delta \sqrt{\log N}, |\Im \log Z_X(\theta)|  \leq A \sqrt{\log N} \right] \leq 16 \, C_0 \,  A \delta + \frac{4 \pi \delta \,  C_0}{\sqrt{\log N}}.$$
Using Lemma \ref{estimateimaginarypart}, one obtains
$$\mathbb{P}_{SU(N)} \left[|\log |Z_X(\theta)| - x_0| \leq \delta \sqrt{\log N} \right] \leq 16 \, C_0 \,  A \delta + \frac{4 \pi \delta \,  C_0}{\sqrt{\log N}} +
c'_1 e^{- \frac{A}{2} \left(A \wedge \frac{\sqrt{\log N}}{2} \right) }.$$
Let us now choose $A := 1 + 5 \log(1/\delta)$.
One gets
$$A \wedge \frac{\sqrt{\log N}}{2}  = [1 + 5 \log(1/\delta)] \wedge \frac{\sqrt{\log N}}{2} \geq \frac{\sqrt{\log 2}}{2}$$
and then
$$ \frac{A}{2} \left(A \wedge \frac{\sqrt{\log N}}{2} \right)  \geq \frac{5 \sqrt{\log 2} \log(1/\delta)}{4} \geq \log(1/\delta).$$
Therefore,
$$ \mathbb{P}_{SU(N)} \left[|\log |Z_X(\theta)| - x_0| \leq \delta \sqrt{\log N} \right] \leq 16 \, C_0 \,   \delta + 80 \, C_0 \, \delta \log(1/\delta)
 + \frac{4 \pi \delta \,  C_0}{\sqrt{\log N}} + c'_1 \delta.$$
Since $\delta < 1/2$, one has $\delta \leq  \delta \log(1/\delta) / \log (2)$, which implies Lemma \ref{boundconcentration}, for
$$ C = \frac{16 \, C_0 }{\log 2} + 80 \, C_0 +  \frac{4 \pi  \,  C_0}{(\log 2)^{3/2}} + \frac{c'_1}{\log 2}.$$

\end{proof}

$ ~~ $
$ ~~ $
$ ~~ $
$ ~~ $

%
%

\subsection{Behaviour of the oscillation in short intervals of the log-characteristic polynomial}

\begin{lemma}
\label{L7BH}  There exists $c_2 > 0 $ such that  for $\mu\in\Rr$ and $A\geq0$ and  uniformly in $N \geq M \geq 2 \vee \frac{|\mu|}{2 \pi}$,

\vspace{-0.3cm}
\begin{eqnarray*}
\Proba{SU(N)}{  \intpi{ \biggabs{ \Re\log Z_X \prth{\theta + \frac{\mu}{N}  }  -  \Re\log Z_X(\theta ) }  } \geq A \sqrt{\log M} } \leq \frac{c_2}{ A^2 },
\end{eqnarray*}

\vspace{-0.3cm}
\begin{eqnarray*}
\Proba{SU(N)}{  \intpi{ \biggabs{ \Im\log Z_X \prth{\theta + \frac{\mu}{N}  }  -  \Im\log Z_X(\theta ) }  } \geq A \sqrt{\log M} } \leq \frac{c_2}{ A^2 }.
\end{eqnarray*}

\end{lemma}

\begin{proof}

By symmetry of the problem, we can assume $\mu > 0$. Setting
 \[
 R_\theta := \Re\log Z_X \prth{\theta + \frac{\mu}{N}  }  -  \Re\log Z_X(\theta )
 \]
 for fixed $ \mu $ (or the same expression with the imaginary part), we get:

\vspace{-0.3cm}
\begin{eqnarray*}
\Proba{SU(N)}{  \intpi{ \abs{ R_\theta }  } \geq A \sqrt{\log M} } & \leq & \frac{1}{ A^2 \log M } \Espr{SU(N)}{ \prth{\intpi{ \abs{ R_\theta }  } }^2 } \\
				& \leq & \frac{1}{ A^2 \log M } \Espr{SU(N)}{ \intpi{  R_\theta^2  } } \\
				& = & \frac{1}{ A^2 \log M } \intpi{ \Espr{SU(N)}{   R_\theta^2  } } \\
				& = & \frac{1}{ A^2 \log M }  \Espr{U(N)}{   R_0^2  }  ~ \mbox{(by \eqref{Disintegration})}
\end{eqnarray*}

Now, under $U(N)$, the canonical matrix $X$ is almost surely unitary: let $\theta_1, \dots, \theta_N$ be its eigenangles in $[0, 2\pi)$. For
$j \in \{1, \dots, N\}$ and $t \in [0, 2 \pi) \backslash \{\theta_j\}$, we can expand the logarithm:
$$\log( 1 -  e^{i(\theta_j-t)}) = - \sum_{k \geq 1} \frac{e^{ik(\theta_j - t)}}{k}$$
as a semi-convergent series. Hence, for $t$ such that $Z_X(t) \neq 0$,
$$\log Z_X(t) = - \sum_{j=1}^{N} \sum_{k \geq 1} \frac{e^{ik(\theta_j - t)}}{k} = - \sum_{k \geq 1} \frac{e^{-ikt}}{k} \tr{X^k}.$$

Thus:

\vspace{-0.3cm}
\begin{eqnarray*} 
\Re\log Z_X(t) \!\!\! &=& \!\!\! -\frac{1}{2} \prth{  \sum_{k \geq 1} \frac{1}{k} e^{-i k t} \tr{X^k} +  \sum_{k \geq 1} \frac{1}{k} e^{i k t} \tr{X^{-k}} } = -\frac{1}{2}   \sum_{k \in \Zz^*} \frac{1}{\abs{k}} e^{-i k t} \tr{X^k} \\
\Im\log Z_X(t) \!\!\! &=& \!\!\! -\frac{1}{2i} \prth{  \sum_{k \geq 1} \frac{1}{k} e^{-i k t} \tr{X^k} -  \sum_{k \geq 1} \frac{1}{k} e^{i k t} \tr{X^{-k}} } = - \frac{1}{2i}   \sum_{k \in \Zz^*} \frac{1}{k} e^{-i k t} \tr{X^k}.
\end{eqnarray*}
Here, the series in $k \in  \Zz^*$ are semi-convergent: more precisely, setting for $K \geq 1$,
$$S_t^{(K)} := -\frac{1}{2}   \sum_{k \in \Zz^*, |k| \leq K} \frac{1}{\abs{k}} e^{-i k t} \tr{X^k},$$
and $$S_t := \Re\log Z_X(t),$$
$S_t^{(K)}$ tends almost surely to $S_t$ when $K$ goes to infinity.

Moreover, one has the following classical result (\cite{DS}): for all $p, q \in \Zz$,
\vspace{-0.3cm}
\begin{eqnarray}\label{EspTrace}
\Espr{U(N)}{ \tr{X^p} \overline{ \tr{ X^q } } } = \Unens{p = q} \abs{p} \wedge N.
\end{eqnarray}

\noindent Hence, for $K, L \geq 1$, $t, u \in \Rr$,

\vspace{-0.3cm}
\begin{eqnarray*}
\Espr{U(N)}{ S^{(K)}_t S^{(L)}_u } &=& \Espr{U(N)}{ \frac{1}{4} \sum_{p,q \in \Zz^*, |p| \leq K, |q| \leq L} \frac{e^{-i(pt+qu) }}{\abs{pq}} \tr{X^p} \tr{X^q} } \\
					 &=& \frac{1}{4} \sum_{p,q \in \Zz^*, |p| \leq K, |q| \leq L}
\frac{e^{ -i(pt + qu)}}{\abs{pq}} \Espr{U(N)}{ \tr{X^p} \tr{X^q} }  \\
					 &=& \frac{1}{4} \sum_{p,q \in \Zz^*, |p| \leq K, |q| \leq L } \frac{e^{-i(pt + qu) }}{\abs{pq}} \Unens{p = -q} \abs{q} \wedge N ~\mbox{(from \eqref{EspTrace} )} \\
					 &=& \frac{1}{4} \sum_{k \in \Zz^*, |k| \leq K \wedge L} \frac{e^{ik(u-t) }}{k^2} \abs{k} \wedge N \\
& = & \frac{1}{2} \sum_{1 \leq k \leq
K \wedge L} \frac{ k \wedge N}{k^2} \prth{\frac{e^{ik (u-t)} + e^{-ik(u-t)}}{2}} \\
					 &=& \frac{1}{2} \sum_{1 \leq k \leq K \wedge L } \frac{ k \wedge N}{k^2} \cos\prth{k (u-t)}.
\end{eqnarray*}
One deduces that
\begin{eqnarray*}
\Espr{U(N)}{ (S^{(K)}_t - S^{(L)}_t)^2 }  &=& \Espr{U(N)}{ (S^{(K)}_t)^2   } + \Espr{U(N)}{ (S^{(L)}_t)^2 } - 2 \, \Espr{U(N)}{ S^{(K)}_t S^{(L)}_t }    \\
					 &=& \frac{1}{2} \sum_{k \geq 1} \frac{ k \wedge N}{k^2} \cos\prth{k (u-t)} \left(\Unens{k \leq K} + \Unens{k \leq L} - 2 \Unens{k \leq K \wedge L} \right) \\
					 &=& \frac{1}{2} \sum_{k \geq 1} \frac{ k \wedge N}{k^2} \cos\prth{k (u-t)} \, \Unens{K \wedge L <  k \leq K \vee L}  \\
					 &\leq& \frac{1}{2} \sum_{k \geq K \wedge L} \frac{ k \wedge N}{k^2},
\end{eqnarray*}
which tends to zero when $K \wedge L$ goes to infinity. Hence, $S^{(K)}_t$ converges in $L^2$ when $K$ goes to infinity, and the limit is necessarily $S_t$.
Therefore,
$$\Espr{U(N)}{ S_t S_u } = \lim_{K \rightarrow \infty} \frac{1}{2} \sum_{1 \leq k \leq K \wedge L } \frac{ k \wedge N}{k^2} \cos\prth{k (u-t)}
= \frac{1}{2} \sum_{k \geq 1} \frac{ k \wedge N}{k^2} \cos\prth{k (u-t)}.$$

The same computation with $ \widetilde{S}_t := \Im\log Z_X( t ) $ gives exactly the same equality:

\vspace{-0.3cm}
\begin{eqnarray*}
\Espr{U(N)}{ \widetilde{S}_t \widetilde{S}_u } = \frac{1}{2} \sum_{k \geq 1} \frac{ k \wedge N}{k^2} \cos\prth{k (u-t)}
\end{eqnarray*}

It is therefore enough to achieve the computations only with $ S_t $. Using this last formula, we can write, with $\alpha = \frac{\mu}{N}$:

\vspace{-0.3cm}
\begin{eqnarray*}
\Espr{U(N)}{ R_0^2 } &=& \Espr{U(N)}{ \prth{S_\alpha - S_0}^2 } = 2 \Espr{U}{S_0^2 - S_\alpha S_0} \\
			  &=& \sum_{k \geq 1} \frac{ k \wedge N}{ k^2} \prth{1 - \cos\prth{k \alpha}} \end{eqnarray*}

We can then develop $ \Espr{U(N)}{ R_0^2 } $:

\vspace{-0.3cm}
\begin{eqnarray*}
\Espr{U(N)}{ R_0^2 } =  \sum_{k \geq 1} \frac{ k \wedge N}{ k^2} - \sum_{k \geq 1} \frac{ k \wedge N}{ k^2}  \cos\prth{ \frac{k\mu}{N}}
\end{eqnarray*}
But we also have
\vspace{-0.3cm}
\begin{eqnarray*}
\sum_{k \geq 1} \frac{ k \wedge N}{ k^2} = \sum_{k = 1}^N \frac{1}{k} + N \sum_{k > N} \frac{1}{k^2} = \log N + \gamma + O(1/N) + N \prth{ \frac{1}{N} + O\prth{ \frac{1}{N^2} } }
\end{eqnarray*}

Moreover we have the following result (\cite{H}, p.37), uniformly on $ \theta \in \crochet{-\pi, \pi} $:

\newcommand{\Ci}[1]{\text{Ci}\prth{#1}}
\newcommand{\Si}[1]{\text{Si}\prth{#1}}

\vspace{-0.3cm}
\begin{multline}\label{ApproxSum}
\sum_{k \geq 1} \! \frac{ k \wedge N}{ k^2} \! \cos (k \theta) = -\log \! \biggabs{ 2 \sin\prth{ \frac{\theta}{2} } \! } + \Ci{N \abs{\theta}} + \cos\prth{N \theta}
- \frac{\pi}{2} N \abs{\theta} + N \theta \, \Si{N \theta}  \\
+ O\prth{ \frac{1}{N} } \ \
\end{multline}

\noindent where:

\vspace{-0.3cm}
\begin{eqnarray*}
\Si{z} \! &:=& \ \ \int_0^z \ \frac{\sin x}{x} dx = \frac{\pi}{2} - \frac{\cos z}{z} + \int_z^{+ \infty} \frac{\cos x}{x^2} dx \\
\Ci{z} \! &:=& \!\!\!-\!\! \int_z^{+ \infty} \! \frac{\cos x}{x} dx = \gamma - \log z + \int_0^z \frac{\cos x - 1}{x} dx
\end{eqnarray*}

Recall also that
\begin{eqnarray*} 
 \int_0^{+ \infty}  \frac{\sin x}{x} dx = \frac{\pi}{2}.
\end{eqnarray*}

Let us denote $ f\prth{\mu} := \log \mu  + \frac{\pi}{2} \mu - \cos \mu - \Ci{\mu} - \mu \Si{\mu} $. We have, for $N$ going to infinity:

\vspace{-0.3cm}
\begin{eqnarray*}
\Espr{U(N)}{R_0^2} &=& \log N + 1 + \gamma  + f(\mu) - \log \mu + \log\biggabs{ 2 \sin\prth{ \frac{\theta}{2} } } + O(1/N) ~\mbox{ (with \eqref{ApproxSum} )} \\
			&=& \log N + 1 + \gamma + f(\mu) - \log \mu + \log\prth{2 \frac{\mu}{2N}\prth{ 1 + O\prth{ \prth{ \frac{\mu}{N} }^2 } } } + O(1/N) \\
			&=& 1 + \gamma + f(\mu) + O\prth{ \prth{ \frac{\mu}{N} }^2 } + O(1/N).
\end{eqnarray*}

Let us now study the behavior of the function $f$.

\vspace{-0.3cm}
\begin{eqnarray*}
f(\mu) &=& \log \mu - \cos \mu + \mu \prth{ \frac{\pi}{2} - \Si{\mu} } - \Ci{\mu} \\
	   &=& \log \mu - \cos \mu + \mu \prth{ \frac{\cos \mu}{\mu} - \int_\mu^{+ \infty} \frac{\cos x}{x^2} dx  } - \prth{\gamma - \log \mu + \int_0^\mu \frac{\cos x - 1}{x} dx } \\
	   &=& - \gamma - \mu \int_\mu^{+ \infty} \frac{\cos x}{x^2} dx + \int_0^\mu \frac{\cos x - 1}{x} dx
\end{eqnarray*}

One has:

\vspace{-0.3cm}
\begin{eqnarray*}
f(\mu) &=& - \gamma - \mu O\prth{ \int_\mu^{+ \infty} \frac{1}{x^2} dx } + O\prth{ \int_0^\mu \left(\frac{1}{x} \wedge 1 \right) dx } \\
	   &=& - \gamma + O\prth{1} + O\prth{ 1 + \log (\mu \vee 1) } = O\prth{ \log \prth{\frac{\mu}{2\pi} \vee 2 } },
\end{eqnarray*}
which implies
 \begin{equation}
\Espr{U(N)}{R_0^2} = O\prth{ \log \prth{\frac{\mu}{2 \pi} \vee 2 } } =  O\prth{ \log M }. \label{RealBigO}
\end{equation}

\end{proof}

\subsection{Control in probability of the mean oscillation of the log-characteristic polynomials}

\def\GSUE{\mbox{GSUE}}

\begin{lemma} \label{L8BH}  For a certain $ n \in \Nn $, let us consider an i.i.d. sequence $ (U_j)_{1 \leq j \leq n} $ of random matrices
following the Haar measure on $ SU(N) $. Let us set:

\vspace{-0.3cm}
\begin{eqnarray*}
L_j(\theta) := \frac{\log\abs{Z_{U_j}(\theta)} }{\sqrt{ \frac{1}{2} \log N }}
\end{eqnarray*}

For $ \delta \in (0,1/2) $, let us consider the random set:

\vspace{-0.3cm}
\begin{eqnarray*}
~\Es_\delta        &:=& \bigcup_{i = 1}^n \prth{ \ensemble{ \theta \in \crochet{0, 2\pi} \ / \ \abs{L_i(\theta)} \geq \delta^{-1}  } \cup \bigcup_{j \neq i} \ensemble{ \theta \in \crochet{0, 2\pi} \ / \ \abs{L_j(\theta) - L_i(\theta)} \leq \delta  } } \\
\end{eqnarray*}

Then, there exists $c_3 > 0$, depending only on $n$, such that for all $N \geq 4$:

\vspace{-0.3cm}
\begin{eqnarray*}
\Espr{}{ \lambda_{2\pi}(\Es_\delta) } \leq c_3 \delta \log(1 / \delta),
\end{eqnarray*}
where $ \lambda_{2\pi} $ denotes the normalised Lebesgue measure on $ \crochet{0, 2\pi} $.

\end{lemma}

\begin{proof}

Using Markov's inequality, one gets:
\begin{align*}
\Espr{SU(N)}{ \lambda_{2\pi}(\Es_\delta) }
&  = \int_0^{2 \pi} \frac{d \theta}{2 \pi} \Proba{SU(N)}{\theta \in \Es_\delta}
\\ &  \leq   \sum_{i=1}^n \int_0^{2 \pi} \frac{d \theta}{2 \pi} \Proba{SU(N)}{\abs{L_i(\theta)} \geq \delta^{-1}}   +  \sum_{1 \leq i \neq j \leq n}
 \int_0^{2 \pi} \frac{d \theta}{2 \pi} \Proba{SU(N)}{\abs{L_j(\theta) - L_i(\theta)} \leq \delta}
\\ & \leq n \int_0^{2 \pi} \frac{d \theta}{2 \pi} \Proba{SU(N)}{\abs{\log{|Z_X(\theta)|}} \geq \delta^{-1} \sqrt{\frac{1}{2} \log N}}
\\ & + n (n-1) \sup_{\theta \in [0,2 \pi], x \in \Rr} \Proba{SU(N)}{\abs{\log{|Z_X(\theta)|} - x} \leq \delta \sqrt{\frac{1}{2} \log N}}
\\ & \leq n c_1 e^{-\frac{\delta^{-1}}{\sqrt{2}} \left(\frac{\delta^{-1}}{\sqrt{2}} \wedge \frac{\sqrt{\log N}}{2} \right)} + n (n-1) C (\delta/\sqrt{2}) \log(\sqrt{2}/\delta)
\end{align*}
Now,
$$e^{-\frac{\delta^{-1}}{\sqrt{2}} \left(\frac{\delta^{-1}}{\sqrt{2}} \wedge \frac{\sqrt{\log N}}{2} \right)}
\leq e^{ -\frac{\delta^{-1}}{\sqrt{2}} \left(\frac{2}{\sqrt{2}} \wedge \frac{\sqrt{\log 2}}{2} \right)}
\leq e^{- \frac{\delta^{-1}}{5}} = O(\delta \log(1/\delta))$$
and
$$\delta \log(\sqrt{2}/\delta) \leq \delta \log(\sqrt{\delta^{-1}}/\delta) = \frac{3 \delta}{2}  \log(1/\delta),$$
which gives Lemma \ref{L8BH}.
\end{proof}

$ ~~ $
$ ~~ $
$ ~~ $

\subsection{Control in expectation of the oscillation of the log-characteristic polynomials on a small period}

 In the sequel, we consider the dimension $N \geq 4$, an integer $K$ such that $2 \leq K \leq N/2$, defined as
a function of $N$ which is equivalent to $N/(\log N)^{3/64}$ when $N$ goes to infinity. We denote
$$ M := N/K \geq 2,$$
which is equivalent to $(\log N)^{3/64}$, and we also define a parameter $\delta \in (0,1/4)$ as a function of $N$, equivalent to
$(\log N)^{-3/32}$ when $N$ goes to infinity.
For $\theta_0 \in [0, 2 \pi]$, we denote, for $0 \leq k \leq K$.
$$ \theta_k := \theta_0 + \frac{2 \pi k}{K} = \theta_0 + \frac{2 \pi k M}{N},$$
and for $0 \leq k \leq K-1$,
$$\Delta := \theta_{k+1} - \theta_k = \frac{2 \pi}{K} = \frac{ 2 \pi  M}{N}.$$
The angle $\theta_0$ is chosen in such a way that the following technical condition is satisfied:

$$\sum_{k = 0}^{K-1} \Ee_{SU(N)} \left( \left| \Im\log Z_{X}(\theta_k + (1- \sqrt{\delta})\Delta ) - \Im\log Z_{X}(\theta_{k}
+ \sqrt{\delta} \Delta) \right| \right) $$ $$\leq
K \Ee_{SU(N)} \left(  \int_0^{2 \pi} \frac{d \theta}{2 \pi}
\left| \Im\log Z_{X}(\theta + (1- \sqrt{\delta})\Delta ) - \Im\log Z_{X}(\theta+ \sqrt{\delta} \Delta) \right| \right).$$
This choice is always possible: indeed, if the converse (strict) inequality were true for all $\theta_0$, then one would get a contradiction by integrating with respect to $\theta_0 \in [0, 2\pi/K)$.
We then define the interval $J := [\theta_0, \theta_0 + 2 \pi) = [\theta_0, \theta_K)$. Note that all the objects introduced here can
be defined only as a  function of $N$. Moreover, by applying Lemma \ref{L7BH} to $\theta + \sqrt{\delta} \Delta$ and
$\mu = N (1- 2 \sqrt{\delta} ) \Delta \leq 2 \pi M$, we deduce that the assumption made on $\theta_0$ implies:
\begin{equation}
\sum_{k = 0}^{K-1} \Ee_{SU(N)} \left( \left| \Im\log Z_{X}(\theta_k + (1- \sqrt{\delta})\Delta ) - \Im\log Z_{X}(\theta_{k}
+ \sqrt{\delta} \Delta) \right| \right)  = O \prth{ K \sqrt{\log M} }. \label{controlImtheta0}
\end{equation}

We can then introduce the \textbf{2-oscillation} of the real and imaginary parts of the log-characteristic polynomial:

\begin{definition}

For $\theta \in J$ and $\mu \in [0,2 \pi M]$, and for the canonical matrix $X \in U(N)$, the $2$-oscillations of $ \Re\log Z_X $ and $ \Im\log Z_X $ are defined by

\vspace{-0.3cm}
\begin{eqnarray*}
\Delta_{\mu} R_{\theta} &:=& \frac{1}{\sqrt{\log(M)}} \Bigabs{ \Re\log Z_X \prth{\theta + \frac{\mu}{N}  }  -  \Re\log Z_X(\theta ) } \\
\Delta_{\mu} I_{\theta} &:=& \frac{1}{\sqrt{\log(M)}} \Bigabs{ \Im\log Z_X \prth{\theta + \frac{\mu}{N}  }  -  \Im\log Z_X(\theta ) }
\end{eqnarray*}

In case of several matrices $(X_j)_{1 \leq j \leq n} $, we denote the corresponding $2$-oscillations by $ \Delta_{\mu} R_{\theta}^{(j)} $ and $ \Delta_{\mu} I_{\theta}^{(j)} $.

\end{definition}

In the sequel, we need to introduce several random sets. The most important ones can be informally described as
follows:
\begin{enumerate}
\item A set $\Ns_1$ of indices $k$ such that the average of the 2-oscillations
$\Delta_{\mu} R_{\theta}$ and $\Delta_{\mu} I_{\theta}$ of the log characteristic 
polynomials for $\theta \in [\theta_k, \theta_{k+1}]$ and $\mu \in [0, 2\pi M]$ is sufficiently small.
\item For $k \in \Ns_1$, a subset $\Gs_k$ of $[\theta_k, \theta_{k+1}]$ for which the average of
 the 2-oscillations with respect to $\mu \in [0, 2\pi M]$ is small enough.
 \item A subset $\Ns_2$ of $\Ns_1$ of "good" indices, such that there exists $\theta \in
 [\theta_k, \theta_{k}+\sqrt{\delta} \Delta]$, both in $\Gs_k$ and $\Es_\delta^c$. This last set,
 introduced in Lemma \ref{L8BH}, corresponds to the fact that the logarithms of the absolute values
 of the characterize polynomials are not too large and not too close from each other: from this last
 condition, we can define the "carrier wave".
  \item For $k \in \Ns_2$, and for some $\theta^*_k \in [\theta_k, \theta_{k}+\sqrt{\delta} \Delta]
  \cap \Gs_k \cap \Es_\delta^c$, a subset $\Ys_k$ of $[0, 2 \pi M]$ such that the 2-oscillations
$\Delta_{\mu} R_{\theta^*_k}$ and $\Delta_{\mu} I_{\theta^*_k}$ are sufficiently small. This condition will ensure
that the carrier wave index corresponding to $\theta = \theta^*_k + \mu/N$ does not depend on
$\mu \in \Ys_k$.
\item From this property, we deduce that, for each pair of consecutive gaps between zeros of the carrier wave,
which are sufficiently large to contain an angle of the form $\theta^*_k + \mu/N$ for
$k \in \Ns_2$ and $\mu \in \Ys_k$ ("roomy gaps"), one can find, with the notation of the introduction, a sign change of $i^N e^{iN \theta /2} F_N(e^{-i \theta})$, and then a zero of $F_N$.
\end{enumerate}
All these sets will be precisely defined in the sequel of the paper, in a way such that their measure is "large"
with "high" probability. The corresponding estimates will then be used to prove our main result.
\begin{lemma}
Let $ \Pp_{SU(N)}^{(n)}$ be the $n$-fold product of the Haar measure on $SU(N)$, $ \Ee_{SU(N)}^{(n)}$ the corresponding expectation,  and $(X_j)_{1 \leq j \leq n}$ the canonical sequence of $n$ matrices in $SU(N)$. Then:
\begin{enumerate}

\item  There exists a random set $ \Ns_1 \subset \intcrochet{0, K- 1} $ such that $ \Ee_{SU(N)}^{(n)}( \abs{\Ns_1} )  \geq (1-\delta) K $ and
 $ \Pp_{SU(N)}^{(n)} $-a.s., $ \forall (j, k) \in \intcrochet{1,n} \times \Ns_1 $, \[
\int_{\theta_k}^{ \theta_{k+1} } \intmean{\mu}{2 \pi M}{  \Delta_{\mu} R_{\theta}^{(j)} } \frac{d\theta}{2 \pi} = O\prth{ \frac{1}{\delta K} }
\]
and
\[
 \int_{\theta_k}^{ \theta_{k+1} } \intmean{\mu}{2 \pi M}{  \Delta_{\mu} I_{\theta}^{(j)} } \frac{d\theta}{2 \pi} = O\prth{ \frac{1}{\delta K} }
\]

\item $ \Pp_{SU(N)}^{(n)} $-a.s., $ \forall k \in \Ns_1 $, $ \exists \, \Gs_k \subset [ \theta_k, \theta_{k+1} ) $ such that $  \lambda_{2\pi}\prth{\Gs_k}
  \geq (1-\delta)/K$ and, $ \forall \theta \in \Gs_k, j \in \intcrochet{1,n} $,

\vspace{-0.3cm}
\begin{eqnarray}\label{OGk}
\intmean{\mu}{2 \pi M}{  \Delta_{\mu} R_{\theta}^{(j)} }  = O\prth{ \frac{1}{\delta^2} } \mbox{ and } \intmean{\mu}{2 \pi M}{  \Delta_{\mu} I_{ \theta}^{(j)} }  = O\prth{ \frac{1}{\delta^2} }
\end{eqnarray}

\end{enumerate}
Here, the implied constant in the $O(\cdot)$ symbols depends only on $n$.
\end{lemma}


\begin{proof}

 By \eqref{RealBigO} and the similar estimate for the imaginary part, we have uniformly (with a universal implied constant),
 $$ \Espr{U(N)}{\prth{ \Delta_{\mu} R_{0} }^2} + \Espr{U(N)}{\prth{ \Delta_{\mu} I_{0} }^2} = O(1).$$
 The
 Cauchy-Schwarz inequality ensures that $$\Espr{U(N)}{\Delta_{\mu} R_{0} } +\Espr{U(N)}{\Delta_{\mu} I_{0} }   = O(1),$$ i.e.\
$$ \int_J \Espr{SU(N)}{ \Delta_{\mu} R_{\theta} +  \Delta_{\mu} I_{\theta}  } \frac{d \theta}{2 \pi} = O(1) ,$$ which implies
$$ \int_J  \intmean{\mu}{2 \pi M}{ \Espr{SU(N)}{\Delta_{\mu} R_{\theta}+\Delta_{\mu} I_{\theta}  }  }  \frac{d \theta}{2 \pi}= O(1) .$$
Splitting the interval $J$ into $K$ equal pieces and applying this estimate to $n$ independent matrices $(X_j)_{1 \leq j \leq n}$ following the
Haar measure on $SU(N)$, one gets

\vspace{-0.3cm}
\begin{eqnarray}
\frac{1}{n}\sum_{j = 1}^n \sum_{k = 0}^{K - 1} \Ee_{SU(N)}^{(n)} \left( \int_{\theta_k}^{ \theta_{k+1} } \intmean{\mu}{2 \pi M}{  (\Delta_{\mu} R_{\theta}^{(j)}  + \Delta_{\mu} I_{\theta}^{(j)} )}
\frac{d\theta}{2 \pi}  \right)  = O(1).
\end{eqnarray}
Applying Markov inequality, we deduce that there exists a universal constant $\kappa > 0$, such that

$$
\Ee_{SU(N)}^{(n)} \left( \card\ensemble{ (j,k) \in \intcrochet{1, n} \times \intcrochet{0, K -1} \Big/  \int_{\theta_k}^{ \theta_{k+1} } \intmean{\mu}{2 \pi M}{  \Delta_{\mu} R_{\theta }^{(j)} } \frac{d\theta}{2 \pi} \geq
  \frac{ \kappa n }{K \delta } } \right)  \leq \frac{\delta K}{2}
$$
and
$$
\Ee_{SU(N)}^{(n)} \left( \card\ensemble{ (j,k) \in \intcrochet{1, n} \times \intcrochet{0, K -1} \Big/  \int_{\theta_k}^{ \theta_{k+1} } \intmean{\mu}{2 \pi M}{  \Delta_{\mu} I_{\theta }^{(j)} } \frac{d\theta}{2 \pi} \geq
  \frac{ \kappa n }{K \delta } } \right)  \leq \frac{\delta K}{2}.
$$

We thus set

\vspace{-0.3cm}
\begin{eqnarray*}
\Ns_1 := \bigcap_{j = 1}^n \Bigg\{ k \in \intcrochet{0, K- 1} \Big/ & & \int_{\theta_k}^{ \theta_{k+1} } \intmean{\mu}{2 \pi M}{  \Delta_{\mu} R_{\theta}^{(j)} } \frac{d\theta}{2 \pi} \leq  \frac{\kappa n}{K\delta}    , \\
 & & \int_{\theta_k}^{ \theta_{k+1} } \intmean{\mu}{2 \pi M}{  \Delta_{\mu} I_{\theta}^{(j)} } \frac{d\theta}{2 \pi} \leq  \frac{\kappa n}{ K \delta }  \Bigg\}
\end{eqnarray*}

and we get:

\vspace{-0.3cm}
\begin{eqnarray}\label{CardN1}
\Ee_{SU(N)}^{(n)}( \abs{\Ns_1} )   \geq (1-\delta) K.
\end{eqnarray}

Now, for $ k \in \Ns_1 $, let us set:

\vspace{-0.3cm}
\begin{eqnarray*}
\Gs_k := [ \theta_k, \theta_{k+1} ) \cap \bigcap_{j = 1}^n \ensemble{  \intmean{\mu}{2 \pi M}{  \Delta_{\mu} R_{ . }^{(j)} } \leq \frac{2 \kappa n^2}{\delta^2} , \intmean{\mu}{2 \pi M}{  \Delta_{\mu} I_{ . }^{(j)} } \leq
 \frac{2 \kappa n^2}{\delta^2} }
\end{eqnarray*}

Applying again Markov inequality, we get that $\Pp_{SU(N)}^{(n)} $-a.s.:

\vspace{-0.3cm}
\begin{eqnarray}\label{CardGk}
\lambda_{2\pi}\prth{\Gs_k}  \geq (1 - \delta) / K.
\end{eqnarray}

\end{proof}


We now define good indices.

\begin{definition}[Good indices] An index $ k \in \intcrochet{0, K- 1}$ is said to be \textbf{good} if :

\begin{enumerate}

\item $ k \in \Ns_1 $,

\item $ \Es_\delta^c \cap \Gs_k \cap [\theta_k, \theta_{k} + \sqrt{\delta} \Delta )\neq \emptyset $

\end{enumerate}

We denote by $ \Ns_2 $ the set of good indices :

\vspace{-0.3cm}
\begin{eqnarray}\label{GoodIndices}
\Ns_2 := \ensemble{ k \in \Ns_1 \, \Big/ \, \Es_\delta^c \cap \Gs_k \cap [\theta_k, \theta_{k} + \sqrt{\delta} \Delta) \neq \emptyset }
\end{eqnarray}

An index is said to be \textbf{bad} if it is not good.

\end{definition}


We then get the following result:
\begin{lemma}
With the notation above, the set of good indices satisfies:
$$\Ee_{SU(N)}^{(n)} ( \abs{\Ns_2}) = K \prth{ 1 - O\prth{ \sqrt{\delta} \log(1/\delta) } },$$
where the implied constant in the $O(\cdot)$ symbol depends only on $n$.
\end{lemma}

\begin{proof}
If $ k \in \Ns_2^c $, either $ k \in \Ns_1^c $, or $ k \in \Ns_1 $ and $ \Es_\delta^c \cap \Gs_k \cap [ \theta_k, \theta_{k} + \sqrt{\delta} \Delta)  = \emptyset $ , i.e.\  $ \Ns_2^c = \Ns_1^c \cup \widetilde{\Ns}_1 $ where:

\vspace{-0.3cm}
\begin{eqnarray*}
\widetilde{\Ns}_1 := \ensemble{ k \in \Ns_1 \, \Big/ \, \Gs_k \cap [\theta_k, \theta_{k} + \sqrt{\delta} \Delta) \subset \Es_\delta }.
\end{eqnarray*}

By \eqref{CardN1}, we have $ \Ee_{SU(N)}^{(n)}( \abs{\Ns_1^c} )\leq \delta K $.

For all $ k \in \widetilde{\Ns}_1 $, we have $ \Es_\delta  \supset  \Gs_k \cap [ \theta_k, \theta_{k} + \sqrt{\delta} \Delta )  $, i.e.\ $ \Es_\delta \supset \bigcup_{k \in \widetilde{\Ns}_1} \Gs_k \cap [ \theta_k, \theta_{k} + \sqrt{\delta} \Delta) $, where the union is disjoint, and thus, $ \lambda_{2\pi}\prth{ \Es_\delta } \geq  \abs{\widetilde{\Ns}_1}  \min_k \lambda_{2\pi}\prth{\Gs_k \cap [\theta_k, \theta_{k} + \sqrt{\delta} \Delta) } $.

$ \Pp_{SU(N)}^{(n)} $-a.s., we have:

\vspace{-0.3cm}
\begin{eqnarray*}
\lambda_{2\pi}\prth{ \Gs_k \cap [ \theta_k, \theta_{k} + \sqrt{\delta} \Delta) } & \geq & \lambda_{2\pi}\prth{ \Gs_k } + \lambda_{2\pi}\prth{ [ \theta_k, \theta_{k} + \sqrt{\delta} \Delta ) } - \lambda_{2\pi} \prth{ [ \theta_k, \theta_{k} + \Delta) }\\
										& \geq & \frac{1}{K} \prth{(1 - \delta) +  \sqrt{\delta}   - 1 }   ~~\mbox{   by \eqref{CardGk}}
\end{eqnarray*}

Now, since $ \delta < 1/4 $, we obtain $ \Pp_{SU(N)}^{(n)} $-a.s.:

\vspace{-0.2cm}
\begin{eqnarray}\label{MeasGkTheta}
\lambda_{2\pi}\prth{\Gs_k \cap  [ \theta_k, \theta_{k} + \sqrt{\delta} \Delta) } \geq \frac{\sqrt{\delta}}{2K}
\end{eqnarray}

This implies that $ \Pp_{SU(N)}^{(n)} $-a.s.:

\vspace{-0.3cm}
\begin{eqnarray*}
\abs{\widetilde{\Ns}_1} & \leq & \frac{2 K}{\sqrt{\delta}} \lambda_{2\pi} \prth{\Es_\delta}
\end{eqnarray*}

Now, by Lemma \eqref{L8BH}, $ \Ee_{SU(N)}^{(n)} \left( \abs{\widetilde{\Ns}_1} \right) = O\prth{K \sqrt{\delta} \log(1/\delta)} $ and then:

\vspace{-0.3cm}
\begin{eqnarray*}
\Ee_{SU(N)}^{(n)} ( \abs{\Ns_2^c}) \leq \Ee_{SU(N)}^{(n)} (\abs{\Ns_1^c} )+ \Ee_{SU(N)}^{(n)} \left( \abs{\widetilde{\Ns}_1} \right) \leq \delta K + O\prth{K \sqrt{\delta} \log(1/\delta)}.
\end{eqnarray*}

\end{proof}

\subsection{Speed of the good oscillation of the log-characteristic polynomials}

\begin{lemma} \label{speedgoodoscillation}

With the notation above, and $\Pp_{SU(N)}^{(n)} $-a.s., $ \forall k \in \Ns_2 $, there exists a random set $ \Ys_k \subset \crochet{0,2 \pi M} $,
and  $ \theta^*_k \in \Es_\delta^c \cap \Gs_k \cap [\theta_k, \theta_{k} + \sqrt{\delta} \Delta ) $, such that

\vspace{-0.3cm}
\begin{eqnarray}\label{SpeedGoodOsc}
 \lambda_{2 \pi M} \prth{ \Ys_k }  =  1 - O\prth{ \delta^{-2} ( \log N )^{-1/4} ( \log M )^{1/2} },
\end{eqnarray}
where $\lambda_M$ is $1/2 \pi M$ times the Lebesgue measure, and for all $ j \in \intcrochet{1, n} $, $\mu \in \Ys_k$,

\vspace{-0.3cm}
\begin{equation*}\tag{\ref{SpeedGoodOsc}}
\Delta_{\mu} R_{\theta^*_k}^{(j)} = O\prth{ \frac{( \log N )^{1/4}}{( \log M )^{1/2} } } \mbox{  and  } \Delta_{\mu} I_{\theta^*_k}^{(j)} = O\prth{ \frac{( \log N )^{1/4}}{( \log M )^{1/2} } }
\end{equation*}
Again, the implied constant in the $O(\cdot)$ symbol depends only on $n$.
\end{lemma}


\begin{proof}

Let $ k \in \Ns_2 $ and $ \theta^*_k \in  \Es_\delta^c \cap \Gs_k \cap [\theta_k, \theta_{k} + \sqrt{\delta} \Delta) $. We set:

\vspace{-0.3cm}
\begin{eqnarray*}
\Ys_k := \bigcap_{j = 1}^n  \ensemble{ \Delta_{.} R_{\theta^*_k}^{(j)} \leq \varepsilon , \ \Delta_{.} I_{\theta^*_k}^{(j)} \leq \varepsilon}
\end{eqnarray*}
where
$$\varepsilon := \frac{( \log N )^{1/4}}{( \log M )^{1/2} }.$$
Applying Markov inequality, we get:

\vspace{-0.3cm}
\begin{eqnarray*}
\lambda_{2 \pi M} \prth{ \Ys_k^c }  & \leq & \lambda_{2 \pi M} \prth{ \bigcup_{j = 1}^n  \ensemble{ \Delta_{.} R_{\theta^*_k}^{(j)} \geq \varepsilon } } +  \lambda_{2 \pi M} \prth{ \bigcup_{j = 1}^n
  \ensemble{ \Delta_{.} I_{\theta^*_k}^{(j)} \geq \varepsilon } } \\
					& \leq &   \frac{2n}{\varepsilon } \max_{1 \leq j \leq n} \, \left( \int_0^{2 \pi M} \Delta_{\mu} R^{(j)}_{\theta^*_k}  \frac{d\mu}{2 \pi M} \vee
 \int_0^{2 \pi M} \Delta_{\mu} I^{(j)}_{\theta^*_k}  \frac{d\mu}{2 \pi M} \right)
 = \frac{1}{\varepsilon} O\prth{ \delta^{-2} },
\end{eqnarray*}
by \eqref{OGk} which gives the announced result.

\end{proof}

\subsection{The number of sign changes} $ $

Let us go back to Theorem \ref{main}. We need to estimate the number of zeros of $F_N$ on the unit circle, or equivalently, the
number of values of $\theta \in J$ such that
the following quantity vanishes:
\begin{equation}
i^N e^{iN \theta /2} F_N(e^{-i \theta}) = i^N e^{iN \theta /2} \sum_{j=1}^n b_j \Phi_{U_{N,j}} (e^{-i \theta}) = \sum_{j=1}^n b_j i^N e^{iN \theta/2} Z_{U_{N,j}} (\theta). \label{realpolycar}
\end{equation}
Using the fact that $U_{N,j} \in SU(N)$, one checks that $i^N e^{iN \theta/2} Z_{U_{N,j}} (\theta)$ is real, and then the number of zeros of $F_N$ on the unit circle is
bounded from below by the number of sign changes, when $\theta$ increases from $\theta_0$ to $\theta_0 + 2 \pi$, of the real quantity given by the right-hand side of \eqref{realpolycar}.
Now, the order of magnitude of $\log |Z_{U_{N,j}} (\theta)|$ is $\sqrt{\log N}$ and more precisely, Lemma \ref{L8BH} informally means that for most values of $\theta$, 
the values of $\log |Z_{U_{N,j}} (\theta)|$ for $1 \leq j \leq n$ are pairwise separated by an interval of length of order $\sqrt{\log N}$.  Hence, one of the terms
in the sum at the right-hand side of \eqref{realpolycar} should dominate all the others. If $j$ is the corresponding index, one can expect that the sign changes
of \eqref{realpolycar} can, at least locally, be related to the corresponding sign changes of $i^N e^{iN \theta/2} Z_{U_{N,j}} (\theta)$, which are associated to
the zeros of the characteristic polynomial $Z_{U_{N,j}}$. This should give a lower bound on the number of sign changes of \eqref{realpolycar}.

This informal discussion motivates the following definition.

\begin{definition}
With the notation of the previous subsections, for all $k \in \Ns_2$,
we define the \textbf{carrier wave index} by:

\vspace{-0.3cm}
\begin{eqnarray*}
j_k := \Arg \max_{j} \ensemble{\Re \log Z_{X_j}(\theta^*_k)},
\end{eqnarray*}
where $\theta^*_k$ is the random angle introduced in Lemma \ref{speedgoodoscillation}. Moreover, we consider the following interval:

\vspace{-0.3cm}
\begin{eqnarray*}
J_k := \crochet{ \theta^*_k, \theta^*_{k} + (1 - \sqrt{\delta}) \Delta }
\end{eqnarray*}
\end{definition}

$ ~~ $
$ ~~ $
$ ~~ $

As $ \theta^*_k \in \Es_\delta^c $, we have $ \forall j \neq j_k $, $ \Re \log Z_{X_{j}}(\theta^*_k) \leq \Re \log Z_{X_{j_k}}(\theta^*_k) - \frac{\delta}{\sqrt{2}} \sqrt{\log N} $. From \eqref{SpeedGoodOsc}, we deduce that $ \forall j \neq j_k  $, $ \forall \mu \in \Ys_k $ :

\vspace{-0.3cm}
\begin{eqnarray}\label{SlowOscillationGoodSize}
\Re \log Z_{X_{j}}\prth{ \theta^*_k + \frac{\mu}{N} } \leq \Re \log Z_{X_{j_k}}\prth{ \theta^*_k + \frac{\mu}{N} } - \frac{\delta}{\sqrt{2}} (\log N)^{1/2} + O\prth{ (\log N)^{1/4} }
\end{eqnarray}
\noindent
Now, since
$$1 / \delta = O((\log N)^{1/10}),$$
with a universal implied constant,
we then get, for a universal $c > 0$,
$$ \frac{\left|Z_{X_{j}}\prth{ \theta^*_k + \frac{\mu}{N} } \right|}{ \left|Z_{X_{j_k}}\prth{ \theta^*_k + \frac{\mu}{N} }  \right|}
\leq \exp \left( - 2 c (\log N)^{0.4} + O\prth{ (\log N)^{1/4} } \right)  \leq \exp \left( - c (\log N)^{0.4}  \right),$$
for $N$ large enough, depending only on $n$. This implies:
\begin{align*}
 \left| \sum_{j \neq j_k} b_j Z_{X_{j}}\prth{ \theta^*_k + \frac{\mu}{N} } \right| & \leq  \frac{\sum_j |b_j|}{\min_j |b_j|} \,
\left|  b_{j_k} Z_{X_{j_k}}\prth{ \theta^*_k + \frac{\mu}{N} } \right| \, \exp \left( - c (\log N)^{0.4}  \right) \\ & \leq \frac{1}{2} \left|  b_{j_k} Z_{X_{j_k}}\prth{ \theta^*_k + \frac{\mu}{N} } \right|
\end{align*}
for $N \geq N_0$, where $N_0$ depends only on $n, b_1, \dots, b_n$.
Hence, for $k \in \Ns_2$, $\mu \in \Ys_k$ and $\theta = \theta^*_k + \mu / N$, the quantity
$$G(\theta) := \sum_{j=1}^n b_j i^N e^{iN \theta/2} Z_{X_j} (\theta),$$
which is $\Pp_{SU(N)}^{(n)}$-a.s. real, has the same sign as its term of index $j_k$.

Theorem \ref{main} is proven if we show that the expectation of number of sign changes of $G(\theta)$ for $\theta \in J$, under $\Pp_{SU(N)}^{(n)}$, is bounded from below
by $N - o(N)$. Hence, it is sufficient to get:
$$\Ee_{SU(N)}^{(n)} \left( \sum_{k \in \Ns_2} \mathcal{S}_k \right) \geq N - o(N),$$
where $\mathcal{S}_k$ is the number of sign changes of $b_{j_k} i^N e^{iN \theta/2} Z_{X_{j_k}} (\theta)$, for
$\theta \in J_k \cap \{ \theta^*_k + \frac{\mu}{N}, \mu \in \Ys_k \}$.

Now, for $k \in \Ns_2$, let $\alpha_{k,1} \leq \alpha_{k,2} \leq \cdots \leq \alpha_{k,\nu_k}$ be the eigenangles, counted with multiplicity, of $X_{j_k}$ in the interval $J_k$.
 The sign of $b_j i^N e^{iN \theta/2} Z_{X_{j_k}}$ alternates between the different intervals $(\alpha_{k,1}, \alpha_{k,2}), (\alpha_{k,2}, \alpha_{k,3}), \dots, (\alpha_{k,\nu_k-1}, \alpha_{k, \nu_k})$. Hence,
for each pair of consecutive intervals containing an angle $\theta = \theta^*_k + \frac{\mu}{N}, \mu \in \Ys_k$, we get a contribution of at least $1$ for the quantity
$\mathcal{S}_k$.

Every element of $J_k$ can be written as $\theta^*_k + \frac{\mu}{N}$, for
$$0 \leq \mu \leq (1 - \sqrt{\delta}) N \Delta \leq N \Delta = 2 \pi M.$$
The Lebesgue measure of the elements of $J_k$ for which $\mu \notin \Ys_k$ is then bounded by
$$\frac{1}{N} \lambda (\Ys_k^{c}) = \frac{2 \pi M}{N} \lambda_{2 \pi M} (\Ys_k^{c}),$$
where $\lambda$ denotes the standard Lebesgue measure. Hence, if an interval $(\alpha_{k,\nu}, \alpha_{k,\nu+1})$ has a length strictly greater than this bound, 
it necessarily contains some $\theta = \theta^*_k + \frac{\mu}{N}$ for which $\mu \in \Ys_k$.
For some $c' > 0$ depending only on $n$, this condition is implied by
$$\alpha_{k,\nu+1} - \alpha_{k, \nu} > c' \frac{M}{N} \delta^{-2} (\log N)^{-1/4} (\log M)^{1/2}.$$
We will say that $ (\alpha_{k, \nu}, \alpha_{k, \nu+1})$ is a \textbf{roomy gap} if this inequality is satisfied, and a \textbf{narrow gap} if
$$\alpha_{k,\nu+1} - \alpha_{k, \nu} \leq c' \frac{M}{N} \delta^{-2} (\log N)^{-1/4} (\log M)^{1/2}.$$
By the previous discussion, $\mathcal{S}_k$ is at least the number of pairs of consecutive roomy gaps among the intervals
$(\alpha_{k,1}, \alpha_{k,2}), (\alpha_{k,2}, \alpha_{k,3}), \dots, (\alpha_{k,\nu_k-1}, \alpha_{k, \nu_k})$. If there is no narrow gap,
the number of such pairs is $(\nu_k -2)_+ \geq \nu_k - 2$. Moreover, if among the intervals, we replace a roomy gap by a narrow gap, this removes at most two
pairs of consecutive roomy gaps. Hence, we deduce, for all $k \in \Ns_2$, that
$$\mathcal{S}_k \geq \nu_k - 2 - 2 \psi_k,$$
where $\nu_k$ is the number of zeros of $Z_{X_{j_k}}$ in the interval $J_k$ and $\psi_k$ the number of narrow gaps among these zeros.
Hence, we get the lower bound:
$$\Ee_{SU(N)}^{(n)} \left( \sum_{k \in \Ns_2} \mathcal{S}_k \right) \geq \Ee_{SU(N)}^{(n)} \left( \sum_{k \in \Ns_2} \nu_k - 2 K - 2 \psi \right),$$
where $\psi$ is the total number of narrow gaps among the zeros in $[0, 2\pi)$ of all the functions $(Z_j)_{1 \leq j \leq n}$.

Now, $\Pp_{SU(N)}^{(n)}$-a.s., for all $k \in \Ns_2$, we have:
\begin{align*}
\nu_k & = \left| \{ \theta \in \crochet{ \theta^*_k, \theta^*_{k} + (1 - \sqrt{\delta}) \Delta }, Z_{j_k} (\theta) = 0\} \right|
\\ & \geq \left| \{ \theta \in \crochet{ \theta_k + \sqrt{\delta} \Delta, \theta_k + (1 - \sqrt{\delta}) \Delta }, Z_{j_k} (\theta) = 0\} \right|
\\ & = \frac{N (1 - 2 \sqrt{\delta}) \Delta}{2\pi}   + \frac{1}{\pi}
 \prth{ \Im\log Z_{X_{j_k}}(\theta_k + (1 - \sqrt{\delta}) \Delta )  - \Im\log Z_{X_{j_k}}( \theta_k + \sqrt{\delta} \Delta) }
\\ & \geq \frac{N}{K}  (1 - 2 \sqrt{\delta}) - \frac{1}{\pi} \sum_{j=1}^n
 \left| \prth{ \Im\log Z_{X_j}(\theta_k + (1 - \sqrt{\delta}) \Delta )  - \Im\log Z_{X_j}( \theta_k + \sqrt{\delta} \Delta) } \right|.
\end{align*}
the second equality coming from Proposition \ref{propositionImLogCardZeros}.

Adding this inequality for all $k \in \Ns_2$, taking the expectation and using \eqref{controlImtheta0} yields the estimates:
$$
 \Ee_{SU(N)}^{(n)} \left( \sum_{k \in \Ns_2} \nu_k  \right) \geq \frac{N}{K}  (1 - 2 \sqrt{\delta}) \Ee_{SU(N)}^{(n)} ( |\Ns_2| )
$$ $$-  \sum_{j=1}^n \sum_{k = 0}^{K-1} \Ee_{SU(N)}^{(n)} \left[ \left| \prth{ \Im\log Z_{X_j}(\theta_k +
 (1 - \sqrt{\delta}) \Delta )  - \Im\log Z_{X_j}( \theta_k + \sqrt{\delta} \Delta) } \right| \right] $$
\begin{align}
  & \geq \frac{N}{K}  (1 - 2 \sqrt{\delta}) K (1 - O(\sqrt{\delta} \log (1/\delta)) + O(K \sqrt{\log M}) \nonumber
\\ & \geq N (1 - O(\sqrt{\delta} \log (1/\delta)) + O \prth{ \frac{ N \sqrt{\log M}}{M} }. \label{estimateterm1}
\end{align}
Moreover,
\begin{equation}
2 K = O (N/M).
\label{estimateterm2}
\end{equation}
It remains to estimate
$$\Ee_{SU(N)}^{(n)}[ 2 \psi] = 2n \Ee_{SU(N)} [ \chi] = 2n \Ee_{U(N)} [\chi],$$
where $\chi$ denotes the number of narrow gaps between the eigenvalues of the canonical unitary matrix $X$. The replacement of
$SU(N)$ by $U(N)$ is possible since the notion of narrow gap is invariant by rotation of the eigenvalues.

Now, the last expectation can be estimated by the following result:
\begin{lemma}
For $N \geq 1$ and $\epsilon > 0$, let $U$ be a uniform matrix on $U(N)$ and let $\chi_{\varepsilon}$ be the number of pairs of eigenvalues of $U$
whose argument differ by at most $\varepsilon/N$. Then, $\Ee[ \chi_{\varepsilon} ] = O(N\varepsilon^3)$.
\end{lemma}
\begin{proof}
 For $\theta_1, \theta_2 \in \Rr$, the two-point correlation density of the eigenvalues of $U$ at $e^{i \theta_1}$ and $e^{i \theta_2}$,
with respect to the uniform probability measure on the unitary group, is given
by
$$\rho(e^{i \theta_1}, e^{i \theta_2}) = N^2 \left[ 1 - \left( \frac{\sin[N (\theta_2 - \theta_1)/2]}{N \sin [(\theta_2 - \theta_1)/2]} \right)^2 \right].$$
Now,
$$N | \sin [(\theta_2 - \theta_1)/2] | \leq N |\theta_2 - \theta_1|/2$$
and then
$$\left( \frac{\sin[N (\theta_2 - \theta_1)/2]}{N \sin [(\theta_2 - \theta_1)/2]} \right)^2 \geq \left(\frac{\sin x}{x} \right)^2$$
for $x=N (\theta_2 - \theta_1)/2$. Now, for all $x \in \Rr$, $|\sin x| \geq \sin |x| \geq  |x| - |x|^3/6$, which implies
$$\left(\frac{\sin x}{x} \right)^2 \geq \left(1 - \frac{x^2}{6} \right)^2 \geq 1 - \frac{x^2}{3}$$
and
$$\rho(e^{i \theta_1}, e^{i \theta_2}) \leq N^2 \left[ 1 - \left(\frac{\sin x}{x} \right)^2 \right]
\leq \frac{N^2 x^2}{3} = \frac{ N^4 (\theta_2 - \theta_1)^2}{6}.$$
Integrating the correlation function for $\theta_1 \in [0, 2\pi)$ and $\theta' :=  \theta_2 - \theta_1 \in [-\varepsilon/N, \varepsilon/N]$
gives:
$$ \Ee[ \chi_{\varepsilon} ] \leq \int_0^{2\pi} \frac{d\theta}{2 \pi} \int_{-\varepsilon/N}^{\varepsilon/N} \frac{d \theta'}{2 \pi}
\frac{ N^4 (\theta')^2}{6} \leq N^4 \int_{-\varepsilon/N}^{\varepsilon/N} (\theta')^2 d \theta' = O \prth{N^4 (\varepsilon/N)^3}.
$$
\end{proof}
\noindent
From this result, applied for
$$ \varepsilon = c' M \delta^{-2} (\log N)^{-1/4} (\log M)^{1/2}$$
we get the estimate:
\begin{equation}
\Ee_{SU(N)}^{(n)}[ 2 \psi] =   O(N\varepsilon^3) = O \prth{ N M^3 \delta^{-6} (\log N)^{-3/4} (\log M)^{3/2}}. \label{estimateterm3}
\end{equation}
The estimates \eqref{estimateterm1},  \eqref{estimateterm2} and \eqref{estimateterm3} imply:
\begin{multline*}
\Ee_{SU(N)}^{(n)} \left( \sum_{k \in \Ns_2} \mathcal{S}_k \right) \\
\geq N \left[ 1 - O\left(\sqrt{\delta} \log(1/\delta)
+ \frac{\sqrt{\log M}}{M}
+M^3 \delta^{-6} (\log N)^{-3/4} (\log M)^{3/2}  \right)\right].
\end{multline*}
From the values taken for $\delta$ and $M$, we get:
$$\sqrt{\delta} \log(1/\delta) = O\prth{(\log N)^{-3/64} \log \log N}, $$
$$ \frac{\sqrt{\log M}}{M} = O \prth{\sqrt{\log \log N} (\log N)^{-3/64}}$$
and
\begin{align*}
M^3 \delta^{-6} (\log N)^{-3/4} (\log M)^{3/2} & = O \prth{ (\log N)^{9/64} (\log N)^{18/32} (\log N)^{-3/4} (\log \log N)^{3/2}} \\ & =
O \prth{  (\log N)^{-3/64} (\log \log N)^{3/2}}.
\end{align*}
Finally, we get
$$\Ee_{SU(N)}^{(n)} \left( \sum_{k \in \Ns_2} \mathcal{S}_k \right) = N \left( 1- O\prth{ (\log N)^{-1/22} } \right),$$
which completes the proof of Theorem \ref{main}.

\subsection*{Acknowledgment}

We thank Brian Conrey and David Farmer for encouraging us to investigate this problem.


\bibliographystyle{amsplain}

\end{document}